\documentclass[11pt,a4paper]{amsart}

\usepackage{amsmath,amssymb,bbm}

\usepackage{graphicx,subcaption}
\usepackage{a4wide}

\theoremstyle{plain}
\newtheorem{theorem}{Theorem}[section]
\newtheorem{lemma}[theorem]{Lemma}
\newtheorem{corollary}[theorem]{Corollary}
\newtheorem{prop}[theorem]{Proposition}

\theoremstyle{definition}
\newtheorem{remark}[theorem]{Remark}
\newtheorem{definition}[theorem]{Definition}
\newtheorem{example}[theorem]{Example}

\newcommand{\NN}{\mathbb{N}}
\newcommand{\RR}{\mathbb{R}}

\newcommand{\TT}{\mathbb{T}}
\newcommand{\ZZ}{\mathbb{Z}}

\newcommand{\cC}{\mathcal{C}}
\newcommand{\cO}{\mathcal{O}}

\newcommand{\ee}{\mathrm{e}}
\newcommand{\ii}{\mathrm{i}}
\newcommand{\leb}{\lambda_{_\mathrm{L}}}

\newcommand{\ts}{\hspace{0.5pt}}
\newcommand{\nts}{\hspace{-0.5pt}}
\newcommand{\dd}{\,\mathrm{d}}

\newcommand{\exend}{\hfill $\Diamond$}
\newcommand{\defeq}{\mathrel{\mathop:}=}

\begin{document}

\title[On a family of singular continuous measures]
{On a family of singular continuous measures\\[2mm]
  related to the doubling map} 

\author[M. Baake]{Michael Baake}
\author[M. Coons]{Michael Coons}
\author[J. Evans]{James Evans}
\author[P. Gohlke]{Philipp Gohlke}

\address{Fakult\"{a}t f\"{u}r Mathematik,
  Universit\"{a}t Bielefeld,\newline \hspace*{\parindent}Postfach
  100131, 33501 Bielefeld, Germany}
\email{mbaake@math.uni-bielefeld.de,pgohlke@math.uni-bielefeld.de}

\address{School of Mathematical and Phys.~Sciences, University of
  Newcastle, \newline \hspace*{\parindent}University Drive, Callaghan
  NSW 2308, Australia}
\email{michael.coons@newcastle.edu.au,james.evans10@uon.edu.au}



\begin{abstract}
  Here, we study some measures that can be represented by infinite
  Riesz products of $1$-periodic functions and are related to the
  doubling map. We show that these measures are purely singular
  continuous with respect to Lebesgue measure and that their
  distribution functions satisfy super-polynomial asymptotics near the
  origin, thus providing a family of extremal examples of singular
  measures, including the Thue{\ts}--Morse measure.
\end{abstract}

\maketitle

\vspace*{-2pt}

\section{Introduction}

The Lebesgue decomposition theorem states that any positive regular
Borel measure $\mu$ on $\RR^n$ has a unique decomposition
$\mu=\mu_{\mathsf{pp}}+\mu_{\mathsf{ac}}+\mu_{\mathsf{sc}}$ relative
to Lebesgue measure, where $\mu_{\mathsf{pp}}$, $\mu_{\mathsf{ac}}$
and $\mu_{\mathsf{sc}}$ are mutually singular as measures.  Here,
$\mu_{\mathsf{pp}}$ is pure point (the Bragg part), while
$\mu_{\mathsf{ac}}$ is an absolutely continuous and
$\mu_{\mathsf{sc}}$ is a singular continuous measure.  We call a
measure \emph{pure} if it has only one of these parts. For example,
with respect to Lebesgue measure, a Dirac measure is purely pure
point, and Lebesgue measure is purely absolutely continuous. Purely
singular continuous measures often arise in the study of dynamical
systems, such as those associated to constant-length substitutions;
see \cite{Q,AS,TAO} for general background.

Since the work by Wiener \cite{Wiener}, the spectrum of a sequence has
been identified as an important quantity, and Mahler \cite{Mahler}
studied the first example of singular continuity, the
\emph{Thue--Morse} (TM) measure, by which we mean, in modern
terminology, the diffraction measure\footnote{This should not be
  confused with the unique, ergodic probability measure on the TM
  shift space (the orbit closure of the TM sequence), which is
  sometimes also called TM measure in the literature.}
$\widehat{\ts\ts\gamma^{}_{t}\ts\ts}$ of the classic TM sequence over
the alphabet $\{-1,1\}$; see \cite{Selecta} for some historical
background. This measure is $1$-periodic, and its restriction to the
fundamental domain $[0,1)$ can be identified with the (dynamical)
spectral measure of maximal type in the orthocomplement of the pure
point sector of the TM shift dynamical system; see \cite{TAO, Q} for
details.  Recall that the TM sequence is the bi-infinite sequence
given by $t(0)=1$, for non-negative $n$ by the recursions $t(2n)=t(n)$
and $t(2n+1)=-t(n)$, and for negative integers by $t(n)=t(-n-1)$. The
corresponding diffraction measure,
$\widehat{\ts\ts\gamma^{}_{t}\ts\ts}$, is the Fourier transform of the
auto\-correlation measure $\gamma^{}_t$,
\[
  \gamma^{}_{t}  \, =
  \sum_{m\in\ZZ} \eta^{}_t(m) \, \delta_{m} \ts ,
\]  
where $\eta^{}_t (m)$ is the volume-averaged \emph{autocorrelation
coefficient}
\[
  \eta^{}_t (m) \, = \lim_{N\to\infty}\frac{1}{2N+1}
  \sum_{n=-N}^N \!  t(n) \, t(n+m) \ts .
\]
These coefficients satisfy $\eta^{}_{t} (0)=1$ and
$\eta^{}_{t} (-m) = \eta^{}_{t} (m)$ for $m\in\NN$, as well as the
recursions $\eta^{}_{t} (2 m) = \eta^{}_{t} (m)$ and
$\eta^{}_{t} (2 m + 1) = - \frac{1}{2} \bigl( \eta^{}_{t} (m) +
\eta^{}_{t} (m+1)\bigr)$ for $m\in\NN_0$; see \cite[Sec.~10.1]{TAO}
and references therein. Perhaps the most interesting property of
$\widehat{\ts\ts\gamma^{}_{t} \ts\ts}$ in our context is the fact that
it can be represented as an infinite Riesz product; see
\cite[Ch.~V.7]{Zygmund} for general background on such measures.
Indeed, as a measure on $\RR$, one has
\[
   \widehat{\, \gamma^{}_{t} \,}  =  \ts \prod_{\ell\geqslant 0}
   \bigl( 1-\cos(2^{\ell+1}\pi (.) )\bigr) ,
\]
which is to be understood as the limit of a vaguely converging
sequence of absolutely continuous measures on $\RR$; see
\cite[Sec.~10.1]{TAO} for a detailed discussion, which is based on the
original work by Mahler and Kakutani, as well as \cite{GL,BGKS,BG} for
results on its scaling properties.  Since this measure is
$1$-periodic, hence of the form
$\widehat{\ts\ts\gamma^{}_{t}\ts\ts} = \mu_{_\mathrm{TM}} \! *
\delta^{}_{\ZZ}$ with
$\mu_{_\mathrm{TM}} \nts = \widehat{\ts\ts\gamma^{}_{t}\ts\ts}\ts
|^{}_{[0,1)}$, it is natural and more convenient to view it as a
finite measure on the $1$-torus, where it is a probability measure,
and work with weak convergence. This is our point of view from now on,
where the autocorrelation coefficients $\eta^{}_{t} (m)$ agree with
the Fourier--Stieltjes coefficients of the measure
$\mu_{_\mathrm{TM}}$ on $\TT$. This simplifies various steps from a
technical perspective, and is perfectly adequate for a complete study,
including hyperuniformity aspects \cite{GL,TS}, which have recently
gained importance in the physical sciences \cite{Aubry,Josh1,BG,Josh2}
and beyond \cite{Grabner,Grabner-2,BC-2,RS}.

Taking the above infinite Riesz product as a starting point, we
investigate the spectral properties of measures that can be
represented as infinite Riesz products
\[
   \prod_{\ell\geqslant 0} h(2^\ell x) \ts ,
\]   
where $h$ is a non-negative continuous function that is $1$-periodic.
Such products are connected with the \emph{doubling map}
$x \mapsto Tx \defeq 2 \ts x$ on the $1$-torus, $\TT$, the latter
represented as $[0,1)$ with addition modulo $1$.  We denote the
corresponding topological dynamical system by $(\TT, T)$.  This
doubling system, from a spectral perspective, is more complicated than
successive multiplication by an integer $\geqslant 3$; see \cite{BC}
for a related example with Stern's diatomic sequence. Another
difficulty emerges from the observation that the Fourier--Stieltjes
coefficients of a general Riesz product do not have such a simple
recursive structure as those of $\mu_{_\mathrm{TM}}$.

To deal with this class systematically, we will thus assume some
additional symmetry conditions on $h$. In particular, we require
$g\defeq h/2$ to be a \emph{$g$-function} in the sense of Keane
\cite{Keane}, where we have
$g (x) + g \bigl( x + \frac{1}{2} \bigr) = 1$ for all $x\in\TT$.
Under this condition, the corresponding Riesz product (if it exists as
a vague limit) is in fact a \emph{g-measure}. Since the seminal work
of Keane \cite{Keane}, $g$-measures have played an important role in
the development of the thermodynamic formalism, pioneered by work of
Ruelle \cite{Ruelle}, Ledrappier \cite{Ledrappier}, Walters
\cite{Walters} and Bowen \cite{Bowen}, to name just a few.  They are
also intimately connected to a class of stochastic processes, known as
\emph{chains with infinite connections} \cite{DF37} or \emph{chains of
  infinite order} \cite{Harris}. We refer to
\cite{BDEG,FGP20,JOP10,Stenflo} and references therein for more on
probabilistic aspects of $g$-measures.  The methods developed in the
context of $g$-measures have found applications in fields such as
diffraction \cite{Keane}, wavelets \cite{CDF92, ConzeRaugi90, FL98},
multifractal analysis \cite{BGKS,Fan97,Olivier99} and learning models
\cite{BurtonKeller93}. More general Riesz products than those
presented above may fall into the class of \emph{$G$-measures}, a
generalisation of $g$-measures that was introduced by Brown and Dooley
in \cite{BD91}; see also \cite{Fan96}.

Given a $g$-function, the question under which condition there is a
\emph{unique} associated $g$-measure that is also the vague limit of
the corresponding Riesz product has attracted considerable attention;
compare \cite{FGP20,JOP10,Walters} among many others. In most cases,
it is assumed that $g$ is \emph{strictly} positive, notable exceptions
are \cite{ConzeRaugi90, Keane, Ledrappier}. Since $g(0) = 0$ for the
TM measure and since the hyperuniformity of the TM sequence depends on
this property, we do \emph{not} want to make such a restriction. For a
few examples of $g$-functions with zeros, illustrating that there may
or may not be a unique $g$-measure, we refer to \cite[Sec.~4]{Keane}
and \cite[Sec.~VII]{ConzeRaugi90}.

In this note, we will not assume prior knowledge on $g$-measures in
order to make it more accessible for readers coming from a
number-theoretic angle. In particular, all the relevant notions will
be introduced in Section~\ref{sec:mug} and the exposition is mostly
self-contained. Some basic results on $g$-measures that by now can be
considered folklore will be proved in an elementary fashion for the
reader's convenience. We point to the relevant literature as we go
along. In the following, we collect some of this `folklore' and state
it in two theorems for easier reference.
 
Here, we denote the space of continuous functions on $\TT$ by
$\cC (\TT)$.  The first theorem builds on a powerful result from
\cite{Keane}; compare \cite{ConzeRaugi90} for an alternative
  approach.

\begin{theorem}\label{thm:mug}
  Let\/ $g \in \cC (\TT)$ be a\/ $g$-function of summable
  variation,\footnote{See Definition~\ref{def:bd} below for this
    notion.}  and assume that one of the following properties holds.
\begin{enumerate}\itemsep=2pt
\item The function\/ $g$ has at most one zero in\/ $\TT$;
\item The function\/ $g$ has only finitely many zeros in\/ $\TT$, none
  of which wanders into a periodic orbit under\/ $T$; or:
\item All zeros of\/ $g$ lie in\/
  $\bigl[ \frac{1}{4}, \frac{3}{4}\bigr]$, sparing at least one of the
  boundary points.
\end{enumerate}
Then, the probability measures on\/ $\TT$ defined by the
densities\/ $g^{}_n (x) \defeq 2^n\prod_{k=0}^{n-1} g(2^k x)$, as\/
$n\to\infty$, converge weakly to a probability measure on\/ $\TT$,
denoted by\/ $\mu_g$, and $\mu_g$ is strongly mixing on $(\TT,T)$.
\end{theorem}

Below, we call a $g$-function \emph{good} when one of the assumptions
(and thus the conclusion) of Theorem~\ref{thm:mug} holds.  Since both
$\mu_g$ and Lebesgue measure $\leb$ are ergodic, invariant probability
measures for $(\TT,T)$, they are either equal or mutually singular.
In fact, by standard properties of $g$-measures, we can fully
characterise the spectral type of $\mu_g$ for good $g$-functions as
follows.

\begin{theorem}\label{thm:sing}
  Let\/ $g\in\cC (\TT)$ be a good\/ $g$-function, and\/ $\mu_g$ the
  associated measure. Then, we are in one of the following three
  cases.
\begin{enumerate}\itemsep=2pt
\item[(\textsf{ac}):] $\mu_g = \leb$ if and only if\/ $g$ is constant
  on\/ $\TT$, which means $g \equiv \frac{1}{2}$.
\item[(\textsf{pp}):] $\mu_g = \delta^{}_{0}$ if and only if\/
  $g \bigl(\frac{1}{2}\bigr) = 0$.
\item[(\textsf{sc}):] $\mu_g$ is singular continuous with respect to\/
  $\leb$ otherwise.
\end{enumerate}  
\end{theorem}

As we are aiming at generalisations of the TM measure, we will impose
another property. We say that a good $g$-function $g$ exhibits
\emph{power-law scaling} if there exist positive constants $c^{}_1$,
$c^{}_2$, $\theta^{}_1$ and $\theta^{}_2$ such that
\[
    c^{}_1 \ts x^{\theta_1} \, \leqslant \, g(x)
    \, \leqslant \, c^{}_2 \ts x^{\theta_2}
\]
holds for all $x \in \bigl[ 0,\frac{1}{2} \bigr]$. This in particular
implies that $g(0) =0$ is the unique zero of $g$ in the interval
$\bigl[ 0,\frac{1}{2} \bigr]$.

\begin{example}\label{ex:guide}
  Guiding examples of such $g$-functions are given by
  $g^{}_t (x)= \frac{1}{2} \bigl(1-\cos(2\pi x) \bigr)$, which induces
  the TM measure, by the tent map
\[
   g^{}_{\wedge} (x) \, = \, \begin{cases}
    2\ts x , & 0 \leqslant x \leqslant \frac{1}{2} \ts , \\
    2 \ts (1-x), & \frac{1}{2} < x < 1\ts , \end{cases}
\]
and by the square root inspired function 
\[
   g^{}_{\sqrt{\phantom{.}}} (x) \, = \, \begin{cases}
   \sqrt{x} \ts , & 0 \leqslant x \leqslant \frac{1}{4} \ts , \\
   1- \sqrt{\big| x-\frac{1}{2}\big| } \ts , & \frac{1}{4} < x
      \leqslant \frac{3}{4} \ts , \\
   \sqrt{1-x} \ts , & \frac{3}{4} < x < 1\ts . \end{cases}
\]
These three functions exemplify the three behaviours of $g$-functions
with power-law scaling at $0$. Indeed, at $x=0$, the derivative of
$g^{}_t$ vanishes, $g^{}_\wedge$ has finite derivative, and
$g^{}_{\sqrt{\phantom{.}}}$ has undefined (infinite) derivative; see
the top row of Figure~\ref{Fig:gs} for graphs of these functions.
\exend
\end{example}

\begin{figure}[ht]
\begin{center}
\begin{subfigure}{0.328\textwidth} 
  \includegraphics[width=\linewidth]{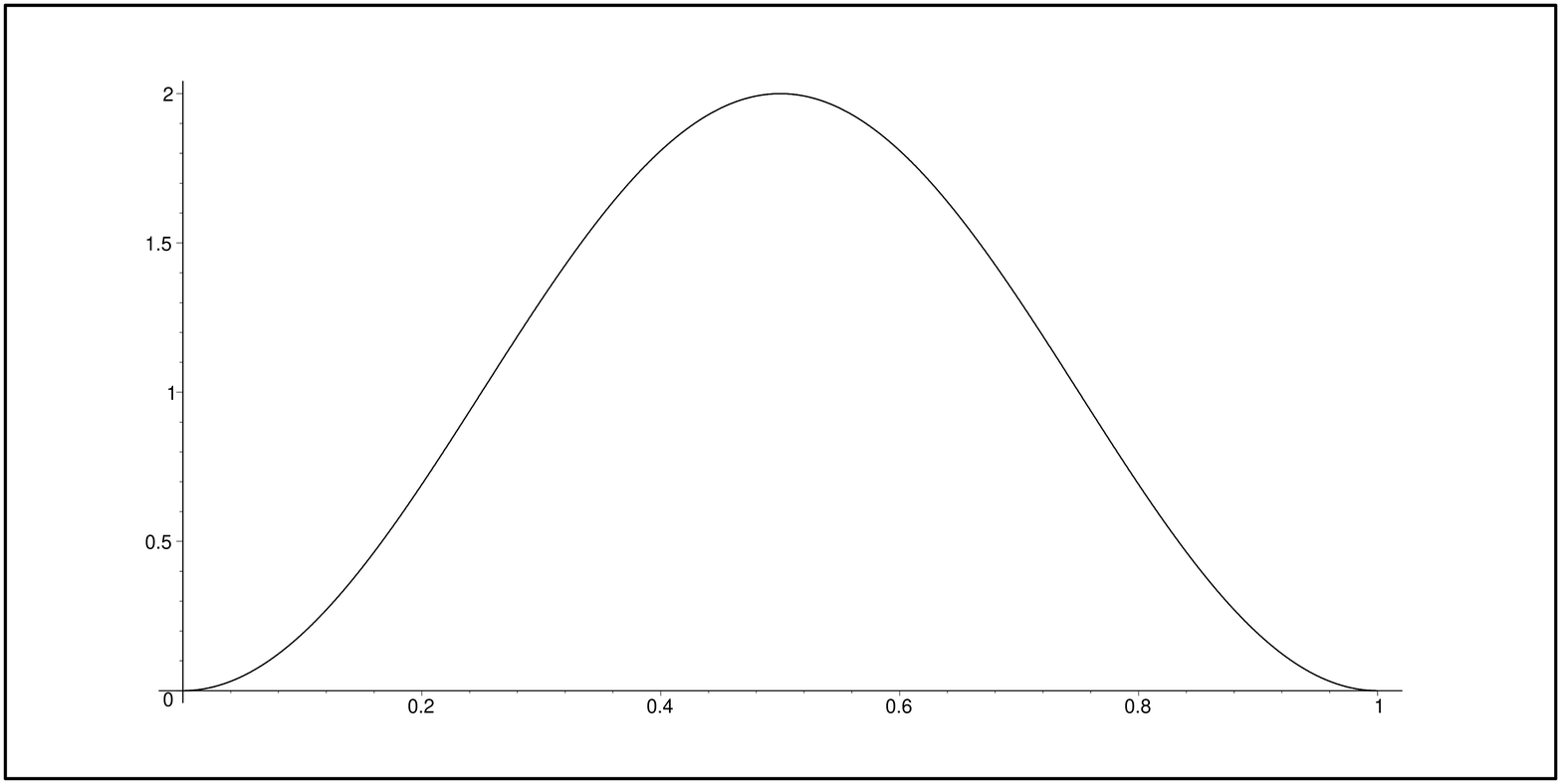} 
  
  \vspace{2pt}
  \includegraphics[width=\linewidth]{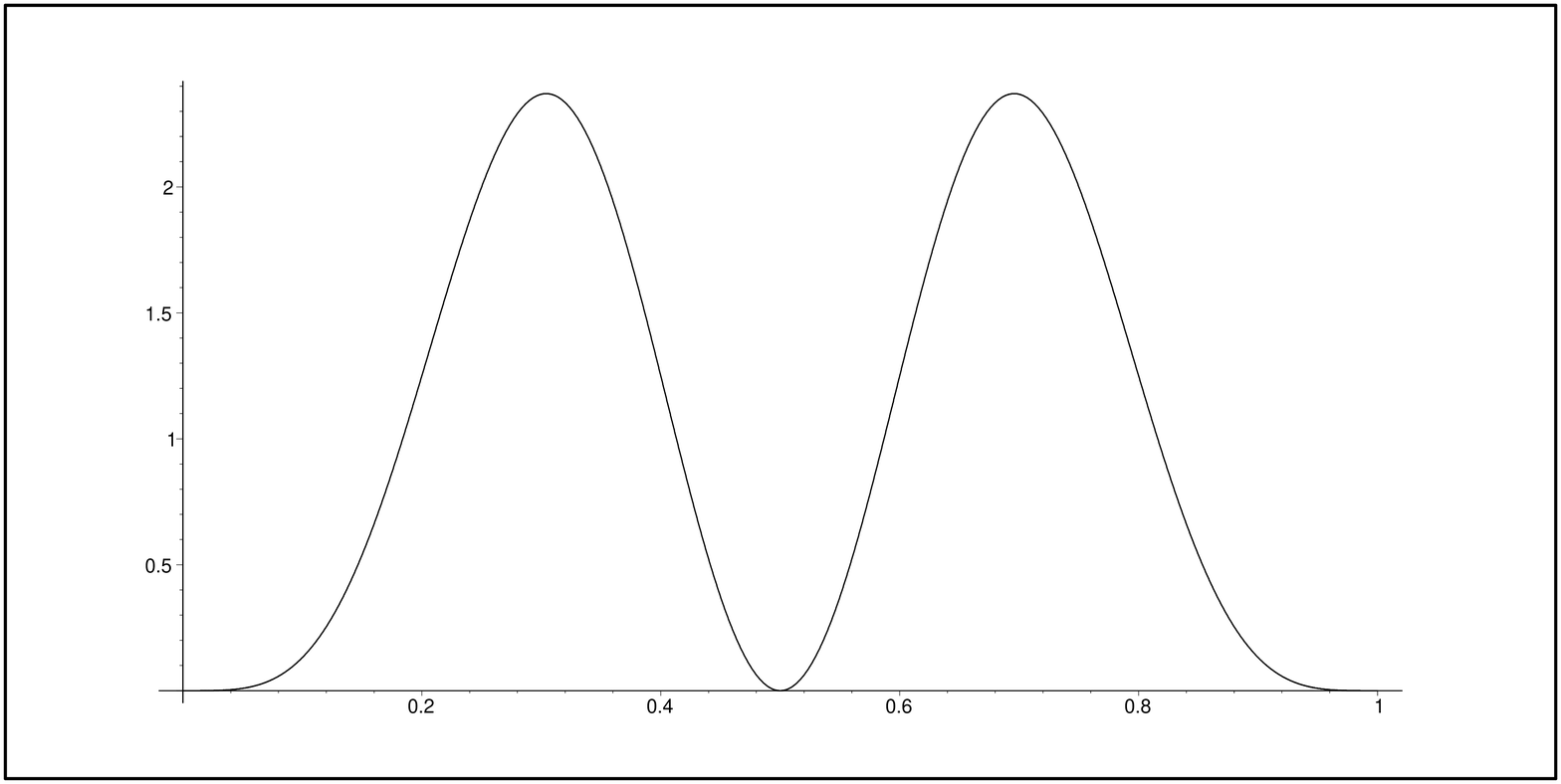} 
  
  \vspace{2pt}
  \includegraphics[width=\linewidth]{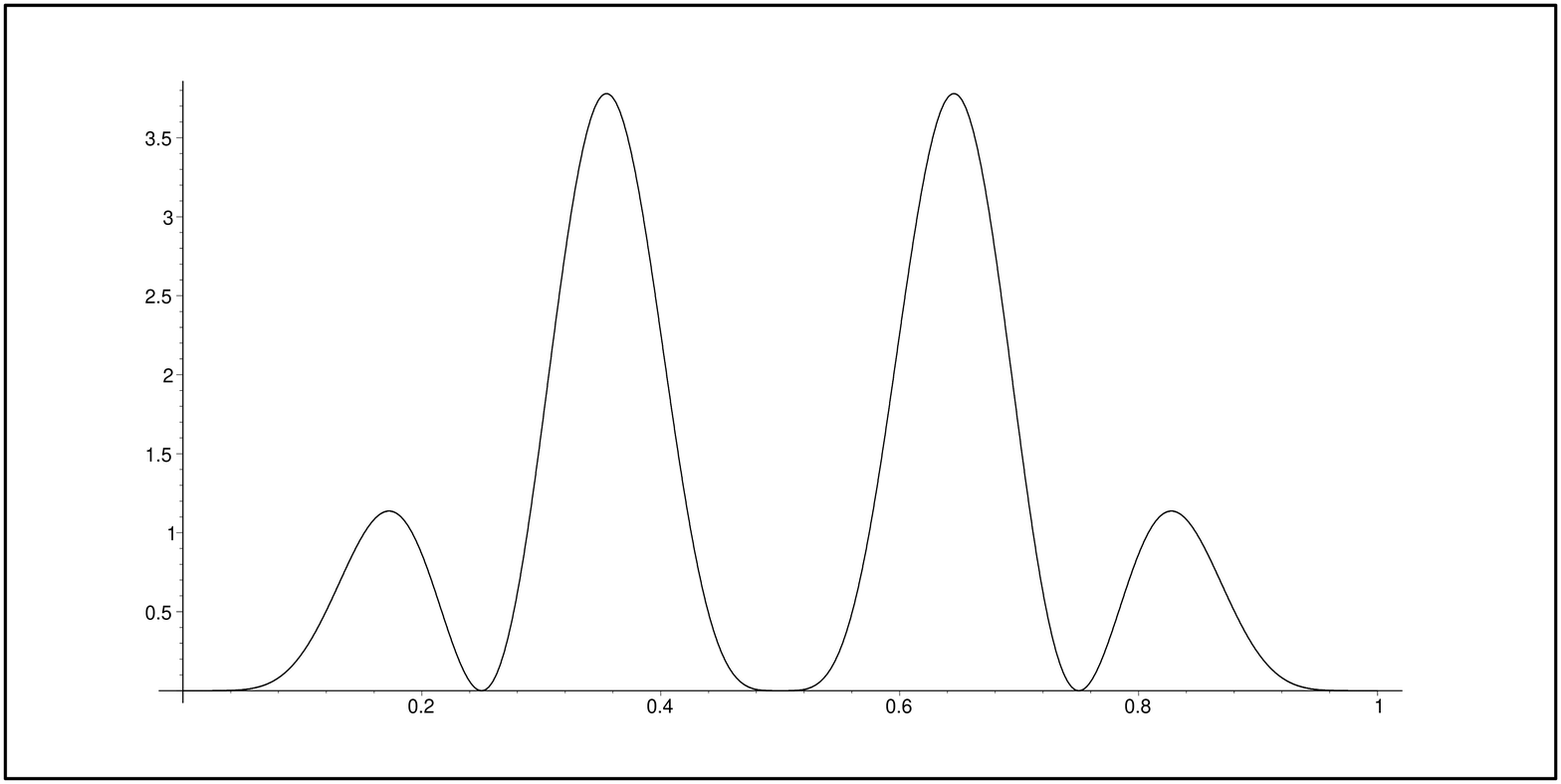} 
  
  \vspace{2pt}
  \includegraphics[width=\linewidth]{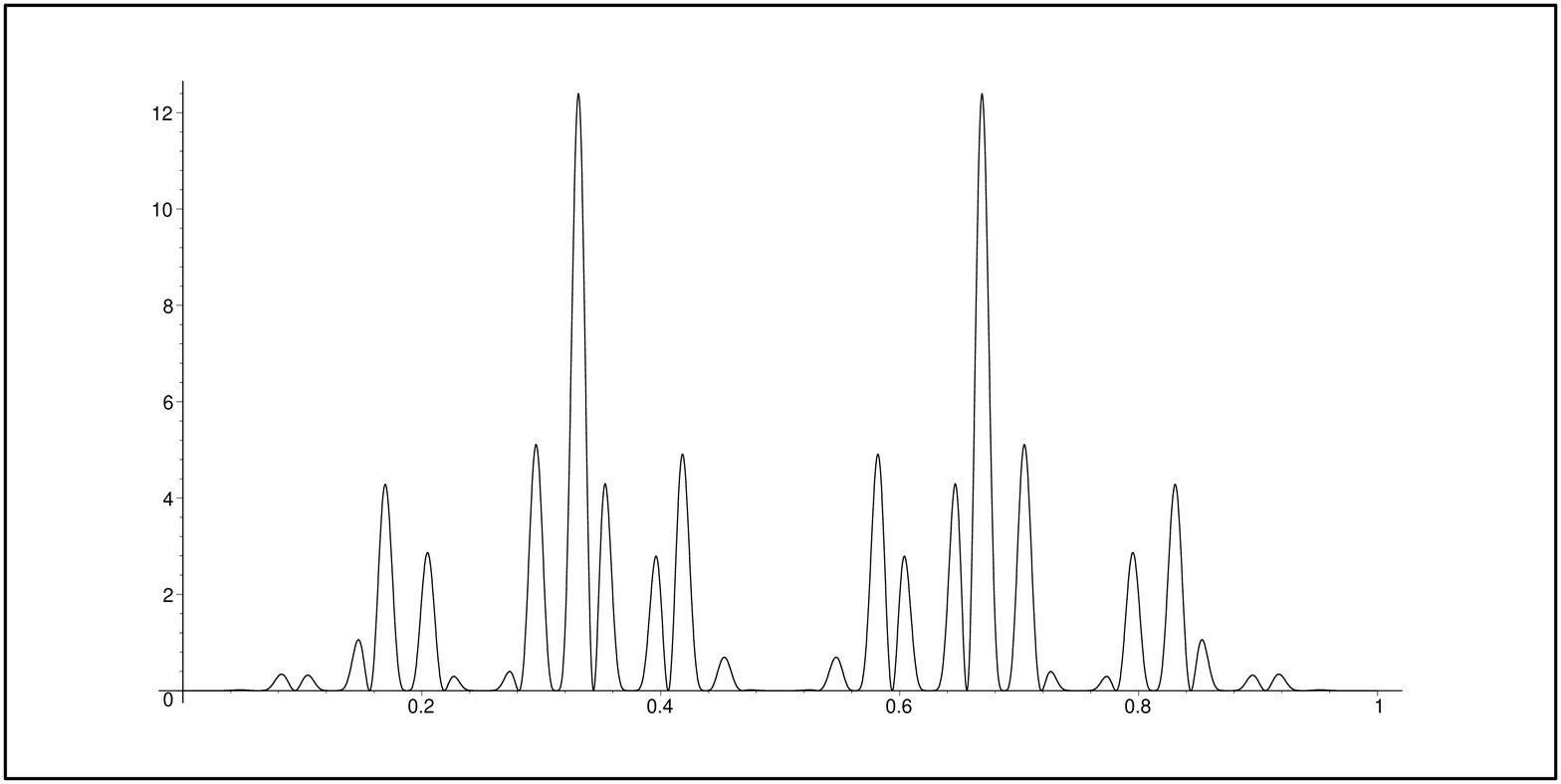}
  
  \vspace{2pt}
  \includegraphics[width=\linewidth]{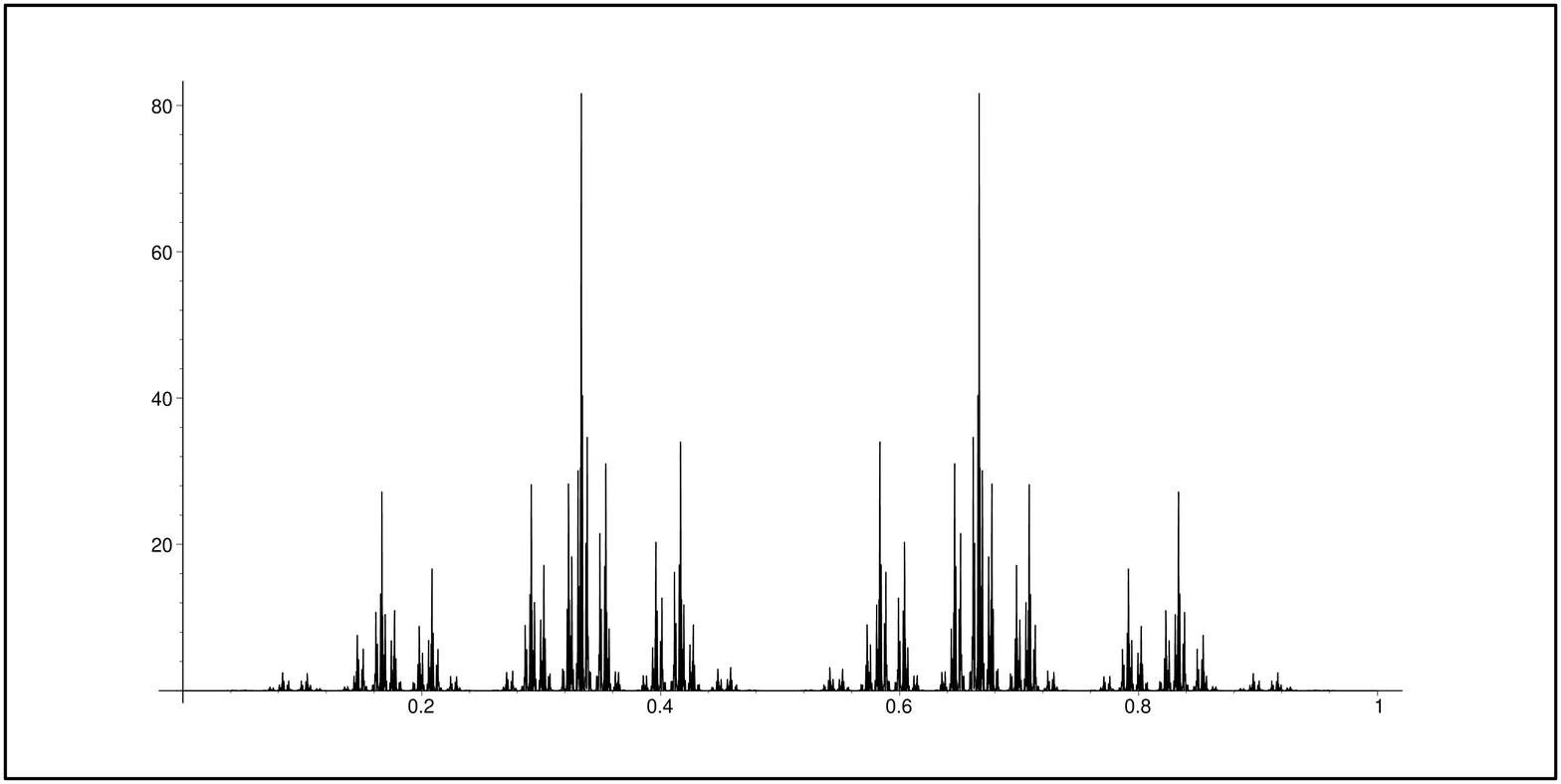}
 
  \vspace{2pt}
  \caption{the function $g^{}_t$}
\end{subfigure}
\begin{subfigure}{0.328\textwidth} 
  \includegraphics[width=\linewidth]{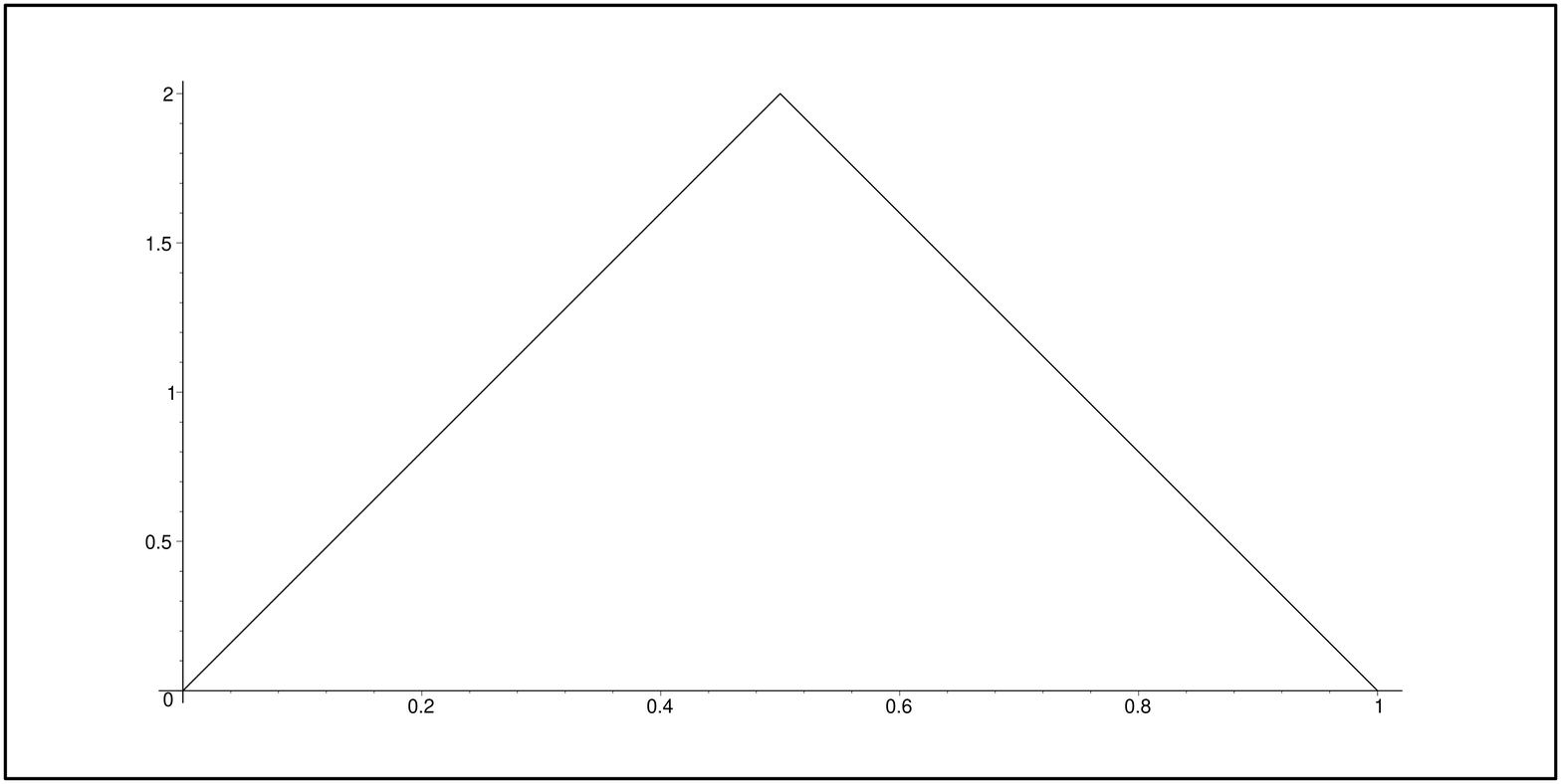} 
  
  \vspace{2pt}
  \includegraphics[width=\linewidth]{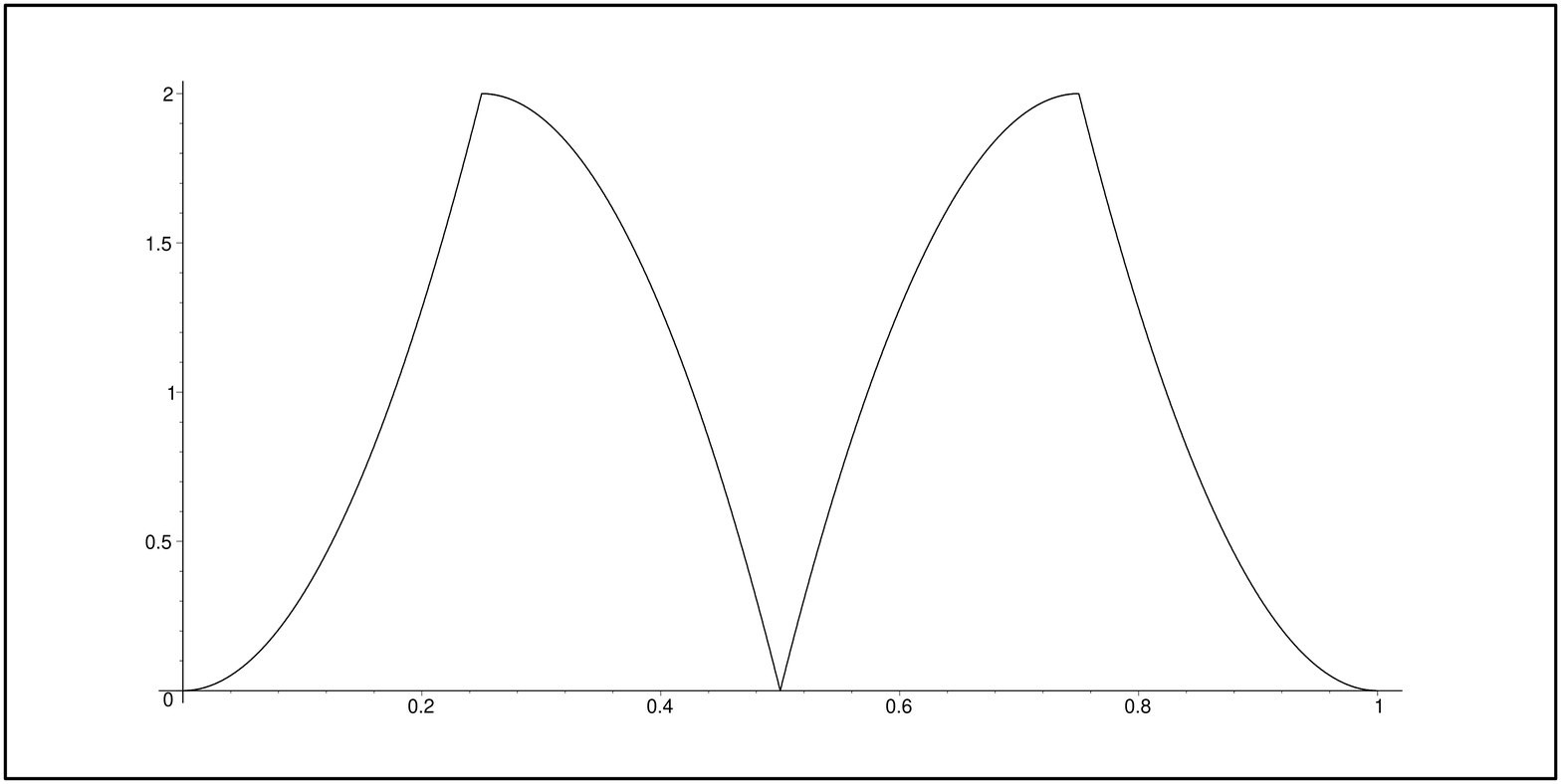} 
  
  \vspace{2pt}
  \includegraphics[width=\linewidth]{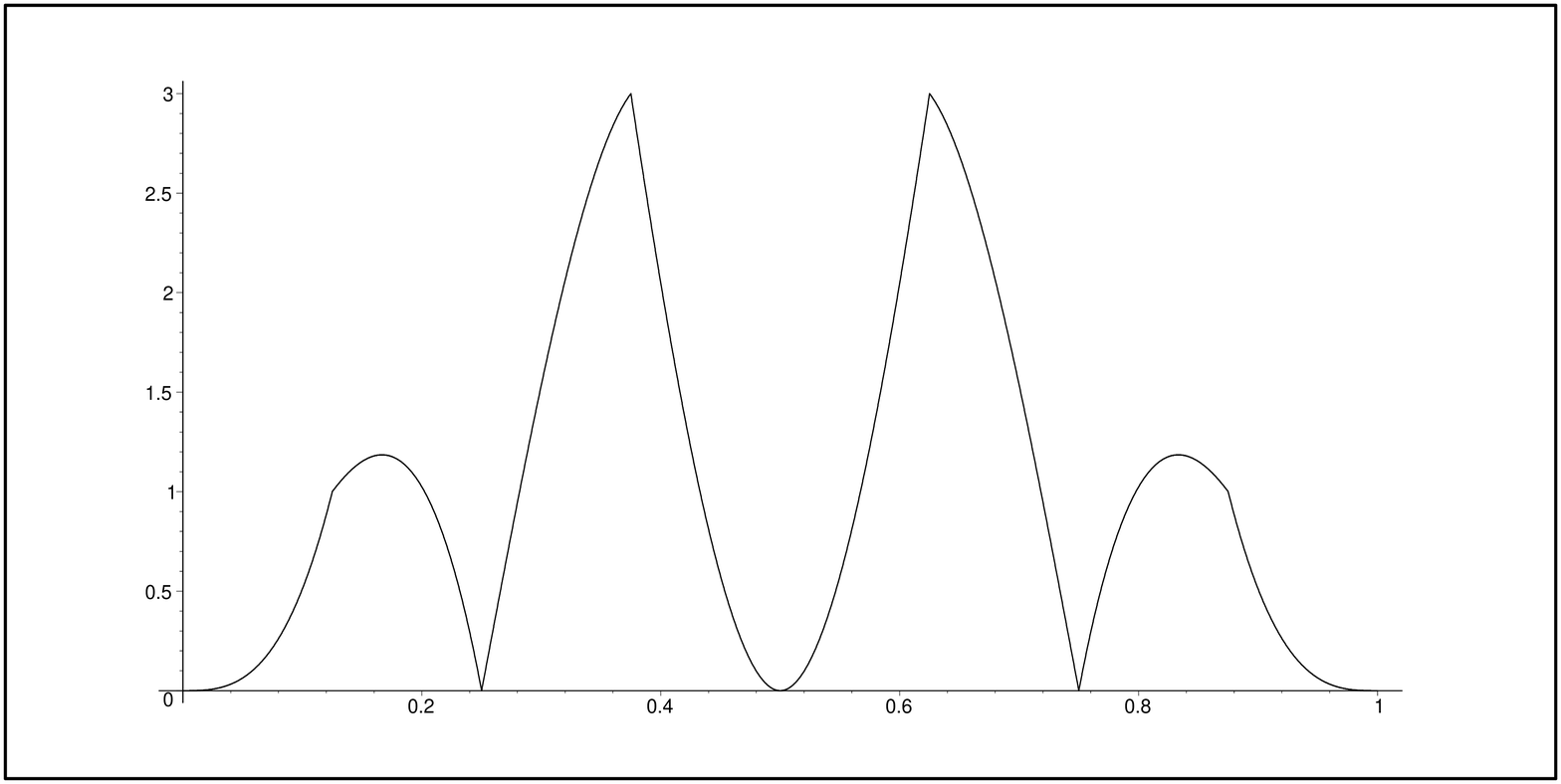} 
  
  \vspace{2pt}
  \includegraphics[width=\linewidth]{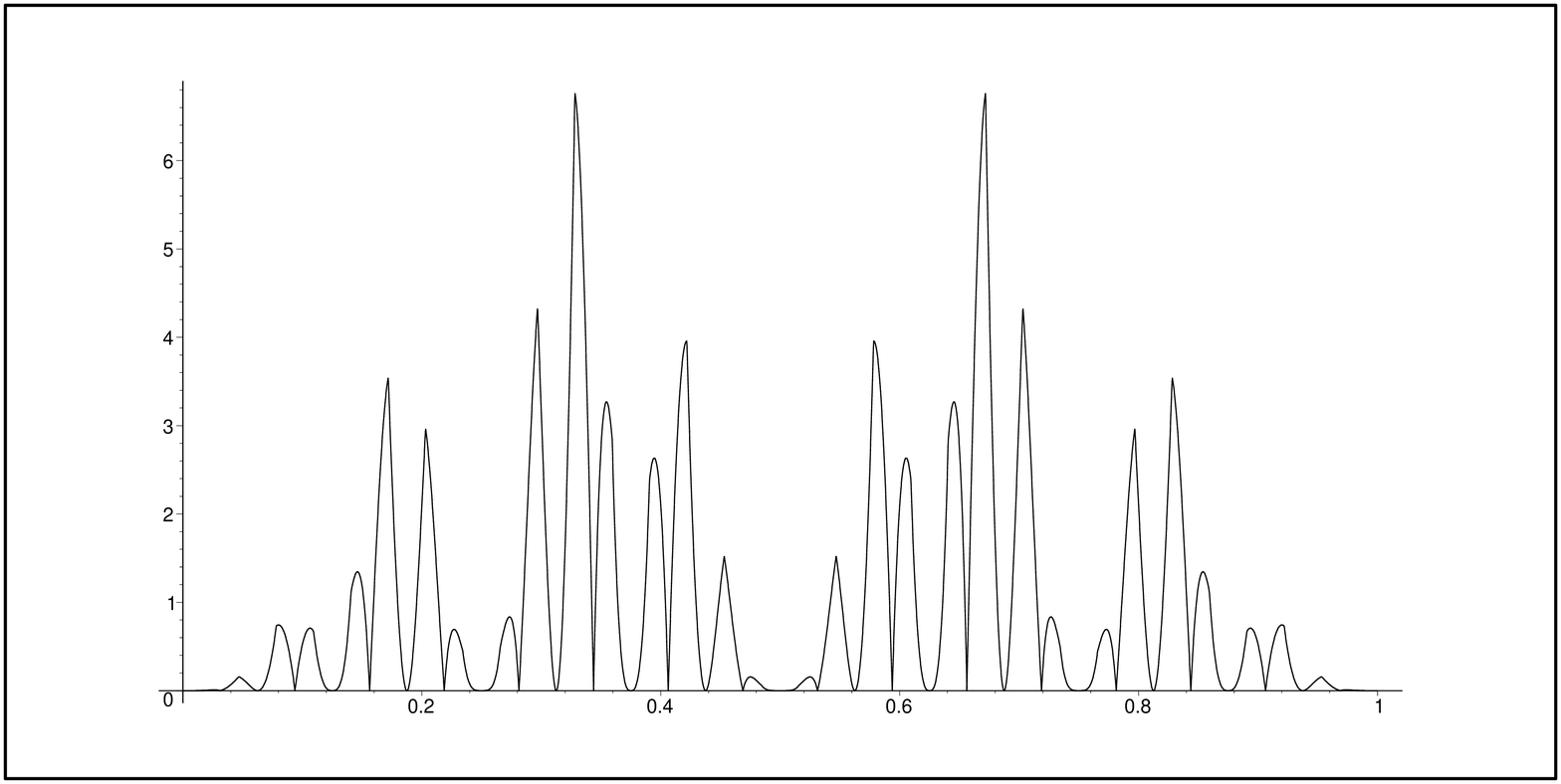} 
  
  \vspace{2pt}
  \includegraphics[width=\linewidth]{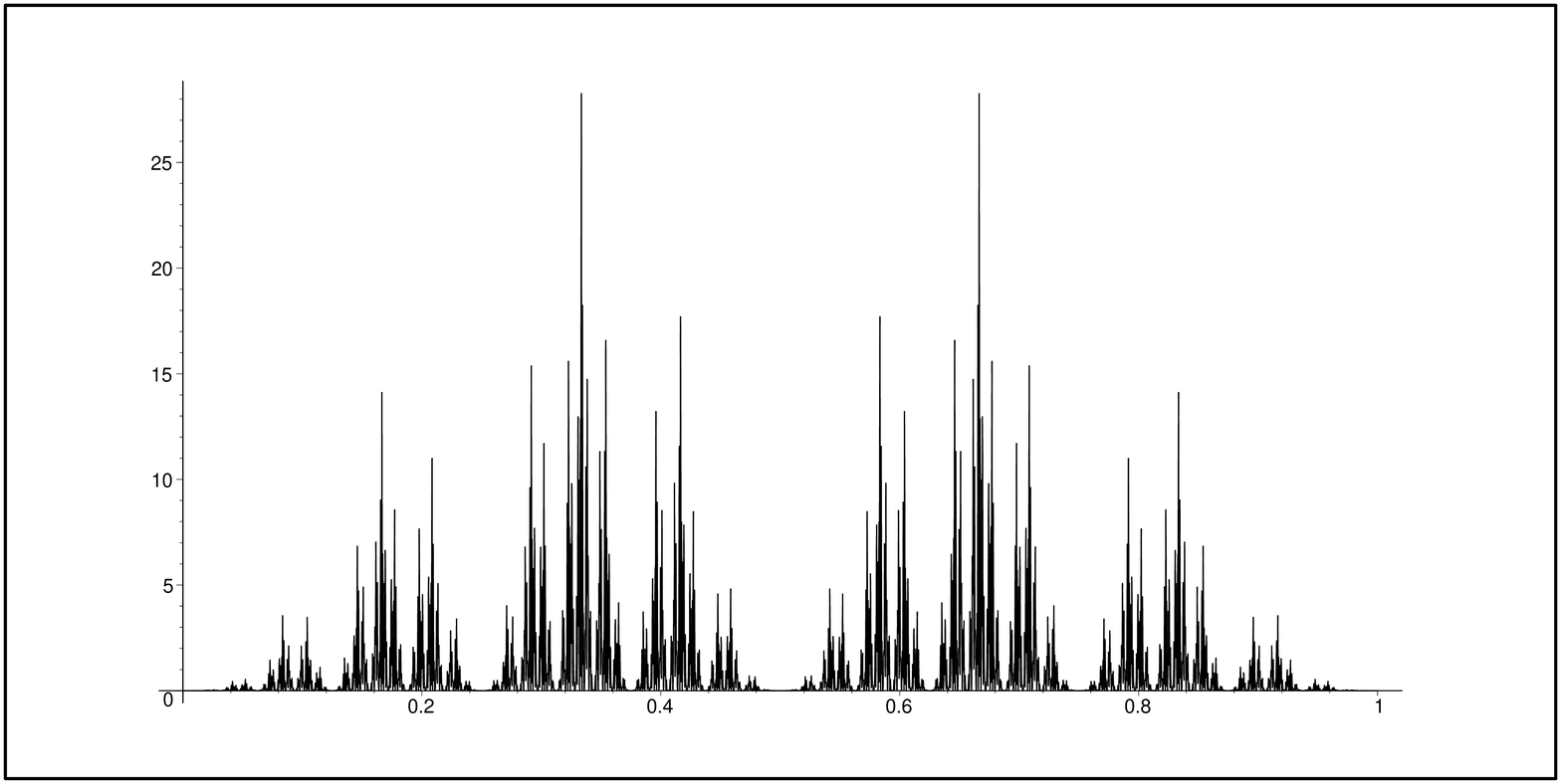}  
 
  \vspace{2pt}
  \caption{the function $g^{}_{\wedge}$}
\end{subfigure}
\begin{subfigure}{0.328\textwidth} 
  \vspace{1pt}
  \includegraphics[width=\linewidth]{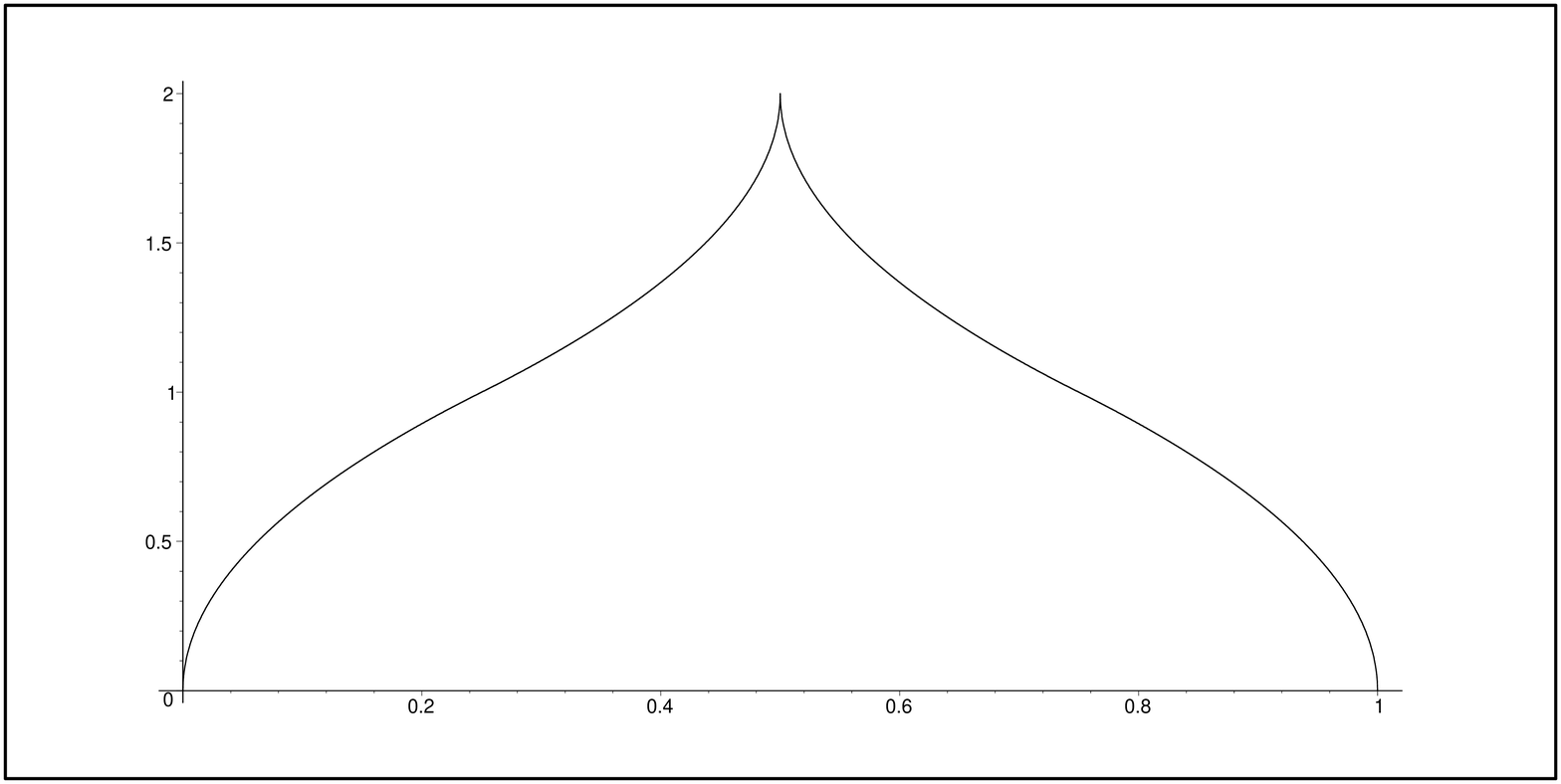} 
  
  \vspace{2pt}
  \includegraphics[width=\linewidth]{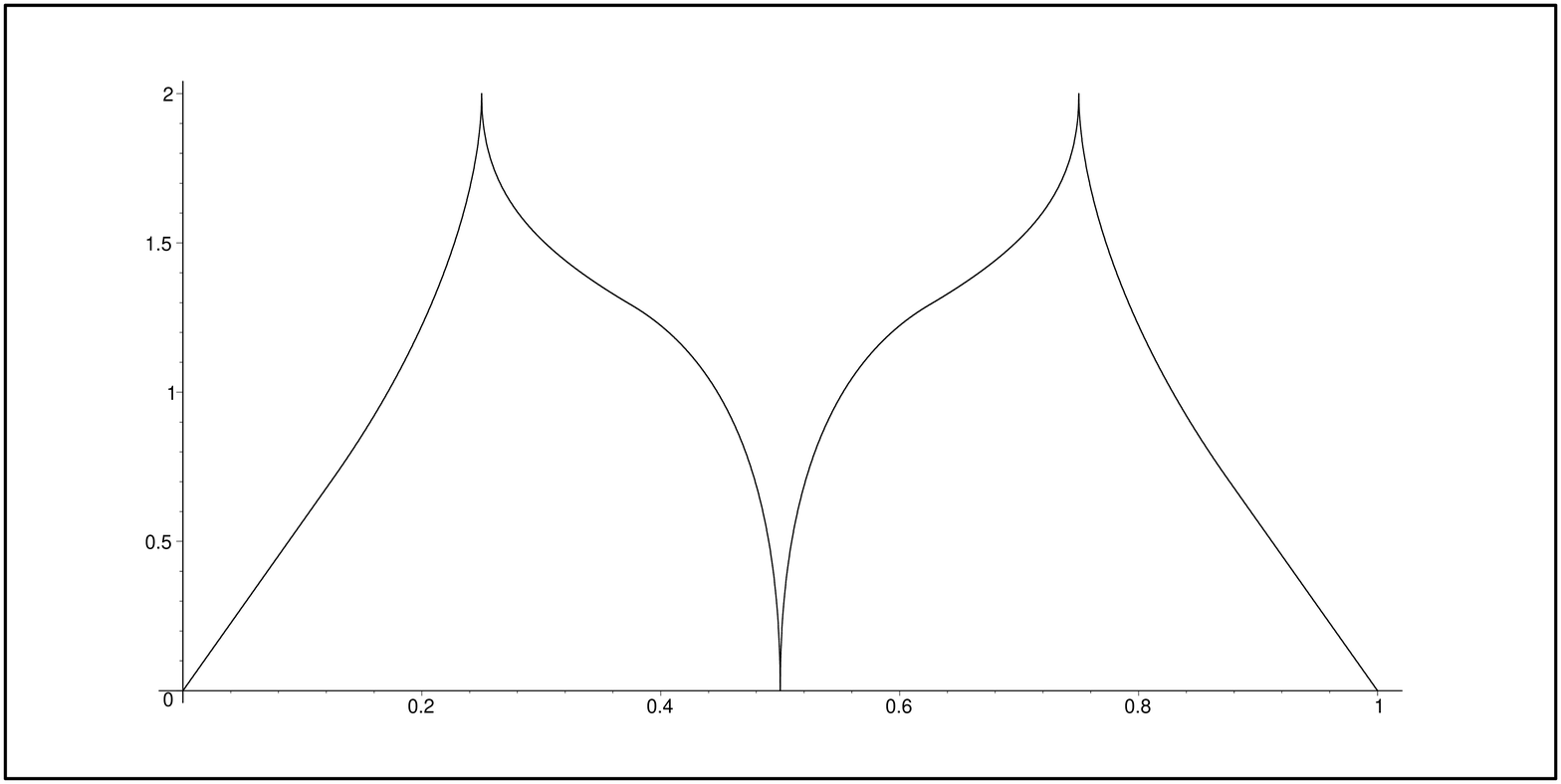}  
  
  \vspace{2pt}
  \includegraphics[width=\linewidth]{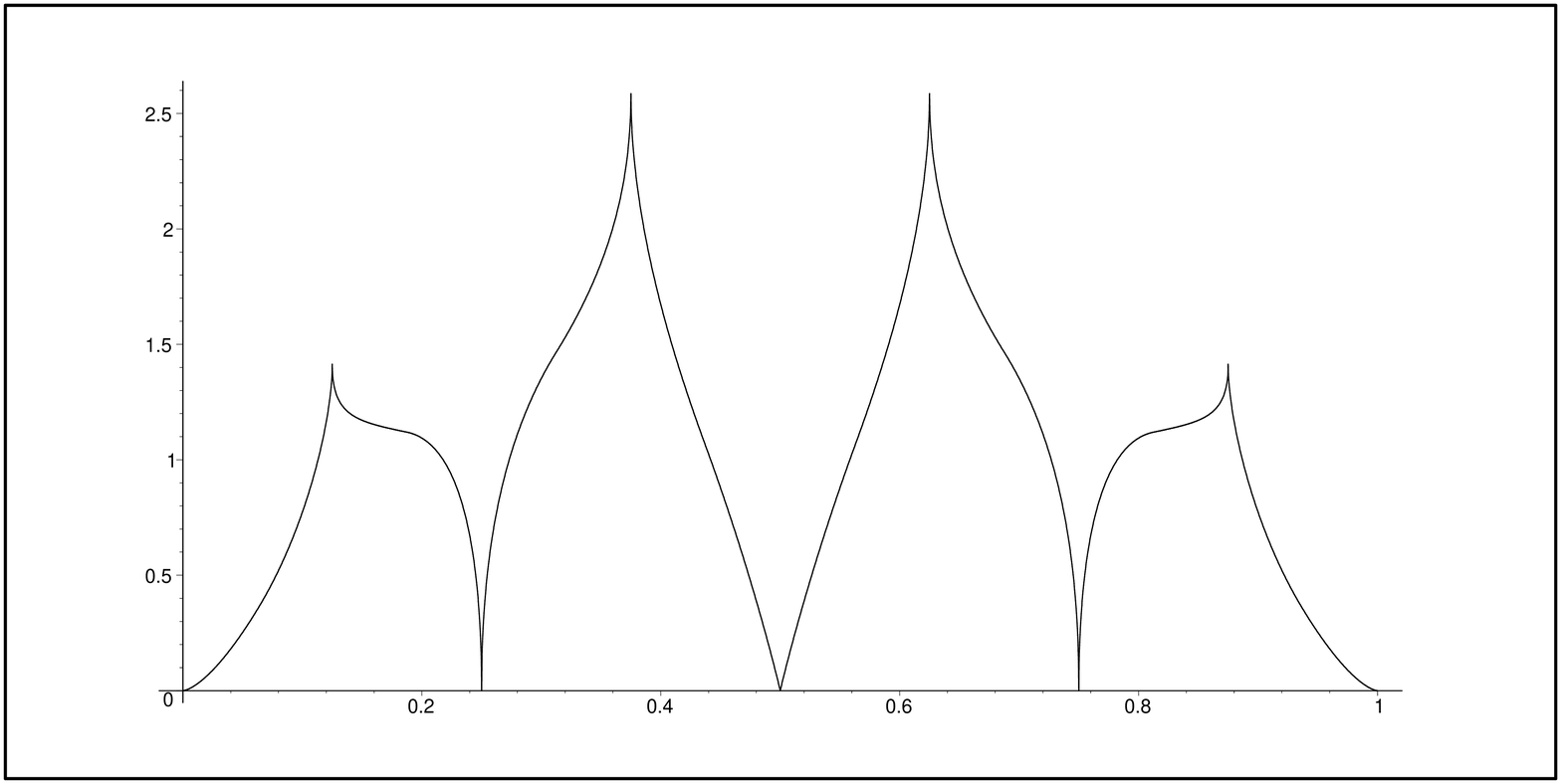}  
  
  \vspace{2pt}
  \includegraphics[width=\linewidth]{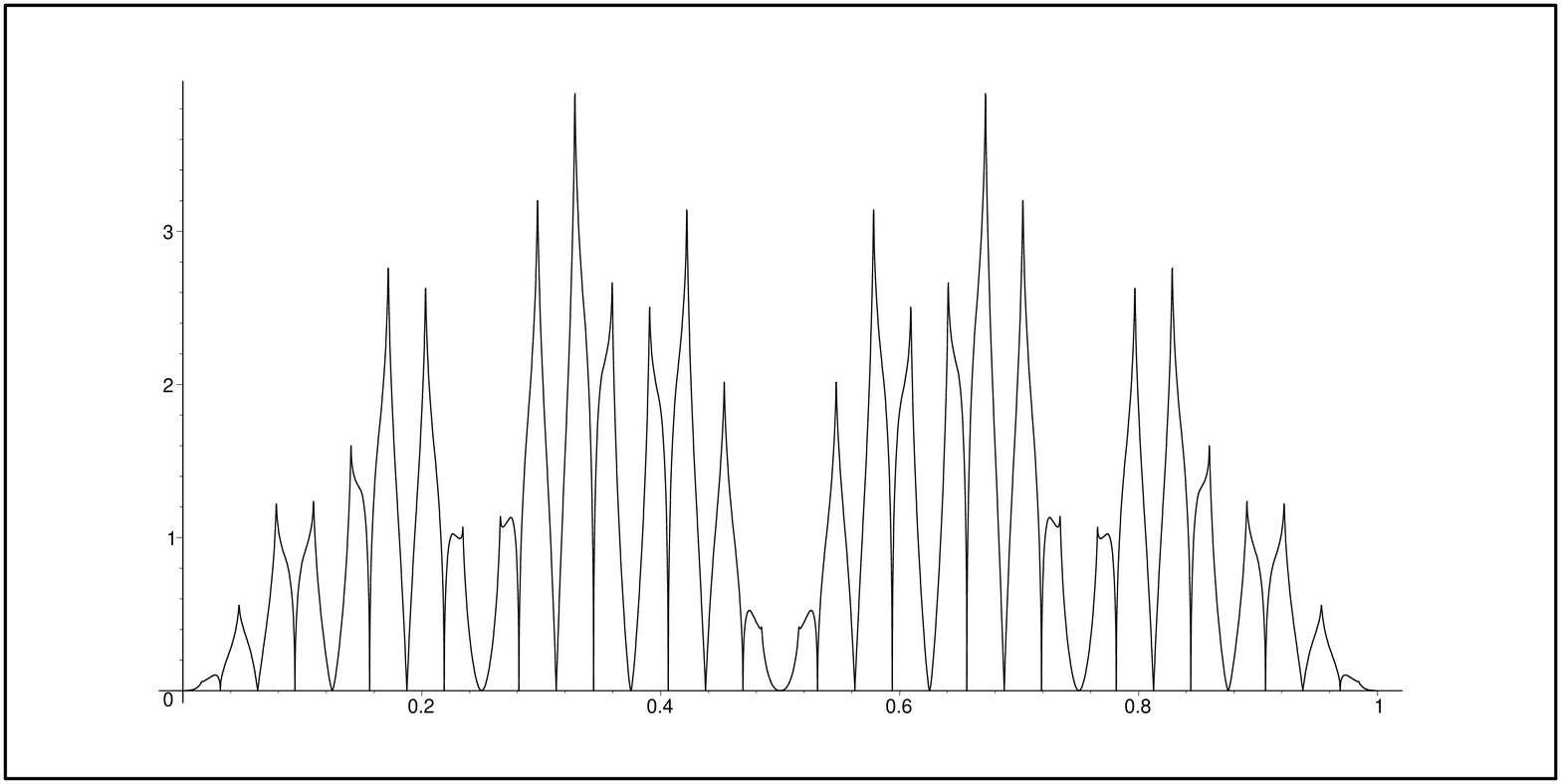}  
  
  \vspace{2pt}
  \includegraphics[width=\linewidth]{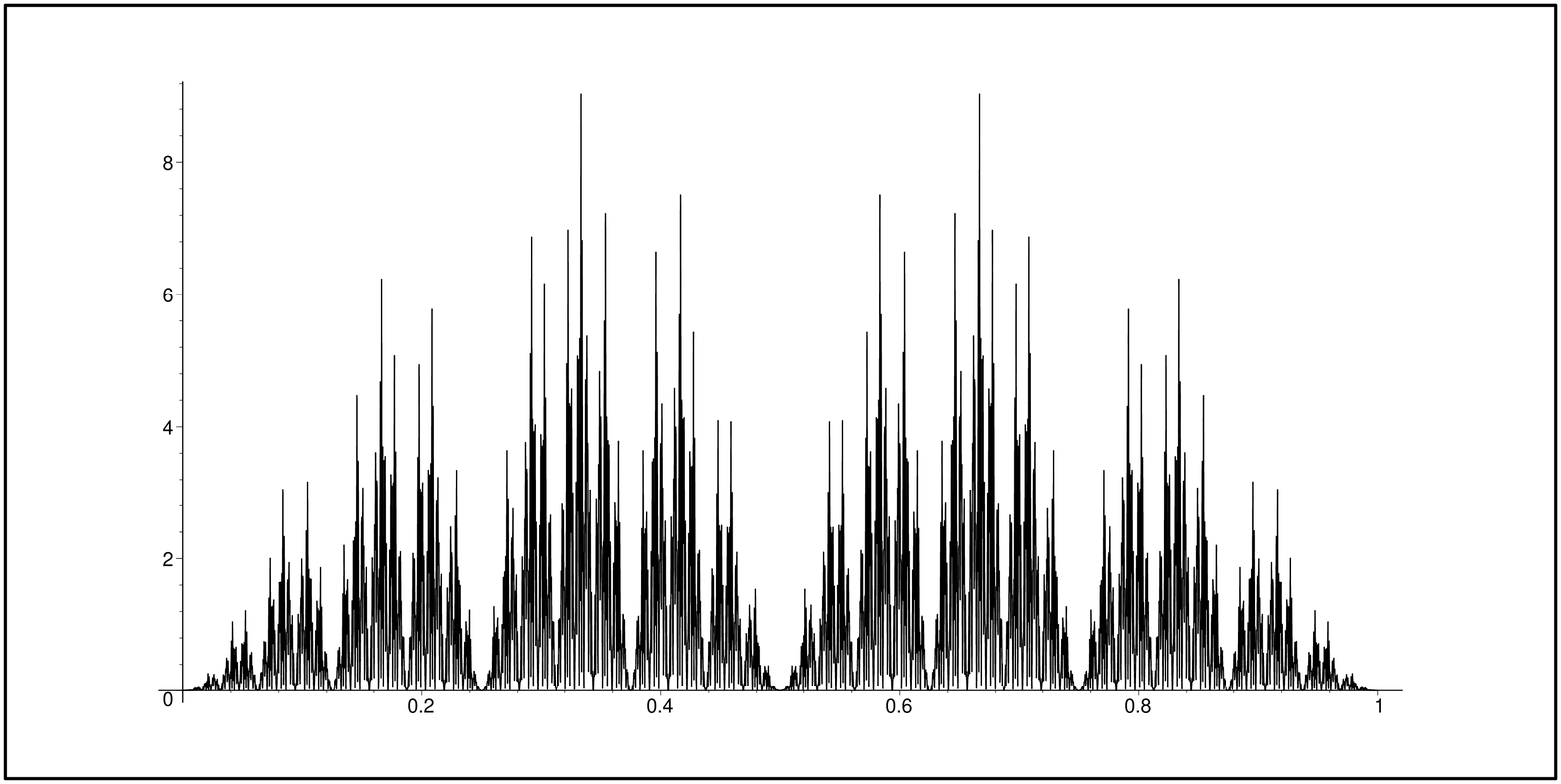}  
 
  \vspace{2pt} 
  \caption{the function $g^{}_{\sqrt{\phantom{.}}}$}
\end{subfigure}
\end{center}
\caption{The probability densities $g^{}_1$, $g^{}_2$,
  $g^{}_3$, $g^{}_6$ and $g^{}_{11}$, from the sequence
  $(g^{}_n)^{}_{n\in\NN}$ that converges to $\mu_g$, for 
  the three $g$-functions of Example~\ref{ex:guide}.}
\label{Fig:gs}
\end{figure}

Next, we examine the distribution function
$F_g (x)\defeq\mu_g([0,x])$. After proving that $F_g (x)$ is strictly
increasing with $x$, which is a generalisation of this known property
in the TM case, we show that the scaling of $F_g (x)$ near $0$ is
super-polynomial. This is our main result.

\begin{theorem}\label{thm:scaling}
  Let\/ $g\in\cC (\TT)$ be a good\/ $g$-function with power-law
  scaling.  Then, one has
\[
   \log \bigl( F_g (x)\bigr) \, \asymp \, -\log^{}_2 (x)^2 
   \quad \text{as } x \to 0^{+} \! ,
\]  
where\/ $\log^{}_2$ denotes the logarithm to base\/ $2$. In
particular, $F_g (x)$ decays faster than any power of\/ $x$ as\/
$x\to 0^+ \!$.
\end{theorem}

This result may be viewed as a first step of a more general
  scaling analysis at arbitrary $x$, as the one in \cite{BGKS} for the
  TM measure, which will be a substantial task for the future. Here,
  we notice that the result of Theorem~\ref{thm:scaling} is of
  particular relevance in number theory, 
  for asymptotical results related to automatic sequences,
  and in the quantitative theory of order and local
  fluctuations, as recently studied in the context of hyperuniformity;
  see \cite{TS,Grabner,Grabner-2,Josh1,Josh2,BG} and references
  therein.  Due to the potentially far-reaching consequences on
  pattern distributions, we feel that the scaling analysis around $0$
  deserves some special attention in its own right.

The remainder of this paper is organised as follows. In
Section~\ref{sec:mug}, using specific properties of $g$-functions, we
derive Theorem~\ref{thm:mug} as a consequence of a more general
result, which is interesting on its own as well.  In
Section~\ref{sec:pp-of-mu_g}, we characterise the pure point part of
general $g$-measures and prove Theorem~\ref{thm:sing}.  Finally,
Section~\ref{sec:dist} contains properties of the distribution
function of $\mu_g$ as well as a proof of Theorem~\ref{thm:scaling},
together with more precise upper bounds on the scaling near $0$, thus
extending a result that is well known \cite{GL,BG} for
$\mu_{_\mathrm{TM}}$.

\section{Some properties of $g$-measures}\label{sec:mug}

As above, let $\TT$ denote the $1$-torus, written as $[0,1)$ with
addition modulo $1$, and let $g \geqslant 0$ be a $g$-function on
$(\TT,T)$, where $T$ is the doubling map $Tx = 2\ts x$. This means
that $g$ is Borel measurable and that we have
$g(x) + g \bigl( x + \frac{1}{2} \bigr) = 1$ for all $x\in\TT$.  Given
a $g$-function on $(\TT,T)$, we define the \emph{transfer operator}
$\varphi_g$ on a real-valued function $f$ on $\TT$ by
\[
  \bigl( \varphi_g f \bigr) (x) \, = \! \sum_{y \in T^{-1} (x)} \!
  g(y) \ts f(y) \, = \, g \bigl( \tfrac{x}{2} \bigr) \ts f \bigl(
  \tfrac{x}{2} \bigr) + g \bigl( \tfrac{x+1}{2} \bigr) \ts f \bigl(
  \tfrac{x+1}{2} \bigr) .
\]
\begin{definition}
  A shift-invariant Borel probability measure $\mu$ on $\TT$ is called
  a \emph{$g$-measure} if $(\varphi_g)_{*} \mu = \mu$ holds, where
  $\bigl((\varphi_g)_{*}\mu\bigr) (f) \defeq \mu (\varphi_g f)$ for
  $f \in \cC (\TT)$.
\end{definition}
The operator $\varphi_g$ preserves the normalisation of a measure,
hence maps probability measures to probability measures.  A short
calculation yields $(\varphi_g)_{*} \leb = 2\ts g \leb $, where $\leb$
is Lebesgue measure on $\TT$ and
$2\ts g \leb(f) = 2 \int_{\TT} g(x) f(x) \dd x$. In particular,
$ \leb$ itself is a $g$-measure precisely if $g \equiv \frac{1}{2}$
almost-surely. More generally, we obtain the following property;  compare \cite[Prop.~1]{FL98}.

\begin{lemma}\label{lem:g_k}
  For\/ $g^{}_k (x) = 2^k \prod_{j=0}^{k-1} g(2^j x)$ and any\/
  $k\in\NN$, one has\/
  $\bigl(\varphi_g^k \bigr){}_{*} \leb = g^{}_k \ts \leb$. In
  particular,\/ $g^{}_k$ is a probability density on $\TT$.
\end{lemma}

\begin{proof}
First, we observe inductively that 
\[
   \bigl(\varphi_g^k f\bigr)(x) \: = \:
    2^{-k} \! \! \sum_{y \in T^{-k}(x)} \! g^{}_k (y) f(y) \ts .
\]
For every $f \in \cC (\TT)$, we now obtain
\[
    \int_0^1 f(x) \, g^{}_k(x) \dd x  \, = 
    \int_0^1 \bigl( \varphi_g^k f \bigr)(x) \dd x \, = \, 
    \bigl(\bigl(\varphi_g^k \bigr){}_{*} \leb\bigr) (f)
\]
by a straightforward calculation. Since
$\bigl(\varphi_g^k \bigr){}_{*} \leb = g^{}_k \ts \leb$ is a
probability measure on $\TT$, it follows that $g^{}_k$ is a
probability density.
\end{proof}

Interpreting $(g^{}_{n})^{}_{n\in\NN}$ as a sequence of probability
measures on $\TT$, by the Banach--Alaoglu theorem, it must have an
accumulation point $\mu_g$ in the weak topology. Assume for a moment
that $g^{}_n$ converges to $\mu_g$. Then, it is easily checked that
$\mu_g$ is in fact a $g$-measure.  Below, we will give a sufficient
condition for weak convergence of $(g^{}_n )^{}_{n \in \NN}$, though
this first needs some preparation via a suitable concept. To formulate
the latter, let $f \in \cC (\TT)$ and set
\[
   f [\delta] \, =  \max_{|x-y| \leqslant \delta} | f(x) - f(y) | 
\]
for $\delta > 0$. Note that the maximum is indeed attained, as $f$ is
uniformly continuous, and that $\lim_{\delta \to 0^+} f[\delta] = 0$.

\begin{definition}\label{def:bd}
   We say that $f \in \cC (\TT)$ is of \emph{summable variation} if
\[
    f^{}_{\delta} \, \defeq \sum_{j=0}^{\infty} 
    f \bigl[ 2^{-j} \delta \bigr] \, < \, \infty \ts ,
\]
for some (equivalently all) $\delta >0$.
\end{definition}

\begin{remark}
  Note that $f[\delta]$ is increasing monotonically in $\delta$.  It
  is straightforward to verify that $f^{}_{\delta} \to 0$ as
  $\delta \to 0^+$ whenever $f$ is of summable variation, and also
  that every H\"{o}lder-continuous function is of summable
  variation. The condition that $f$ is of summable variation appears
  also under the term \emph{Dini continuous} in the literature.
    For the comparison of our results with the literature on
    $g$-measures, it is noteworthy that, if $f>0$, it is of summable
    variation precisely if $\log(f)$ has this property.  \exend
\end{remark}

Even in the case that $g > 0$, additional regularity assumptions
  are needed to conclude that there is a \emph{unique} $g$-measure
  associated to it; see \cite{BK93} for an example that shows that
  continuity alone does not suffice. Summable variation has become a
  standard assumption, and indeed guarantees both uniqueness and
  uniform convergence of $\varphi_g^n f$ to a constant when $g>0$;
  compare \cite{Walters}. It does not immediately yield such a strong
  conclusion if $g$ is allowed to have zeros, but has proved useful
  also in this more general setting \cite{ConzeRaugi90}. The reason is
  that summable variation guarantees that, for each $f \in \cC(\TT)$,
  the family $\{ \varphi_g^n f \}_{n \in \NN}$ is uniformly
  equicontinuous. Indeed, a straightforward calculation shows that
\[
   \bigl( \varphi_g f \bigr) [\delta] \, \leqslant  \,
    2 \ts | f | \, g \bigl[ \tfrac{\delta}{2} \bigr] 
    + f \bigl[ \tfrac{\delta}{2} \bigr] .
\]
from which equicontinuity is quickly deduced if $g$ has summable
variation. Hence, $\{ \varphi_g^n f \}_{n \in \NN}$ has a uniformly
converging subsequence by the Arzela--Ascoli theorem.

\begin{remark}
  The term \emph{summable variation} is somewhat more natural in the
  context of shift spaces. This is the classic setup in the context of
  the thermodynamic formalism \cite{Ruelle, Walters}. In fact, there
  is a natural identification of $(\TT,T)$ with the one-sided shift
  $(\mathbb{X},S)$, with the space $\mathbb{X} = \{0,1\}^{\NN}$ 
  and with $S$ denoting the left shift action, $(Sx)_n = x_{n+1}$ 
  for all $n \in \NN$. The identification is given by the semi-conjugation
  $(x_n)_{n \in \NN} \mapsto \sum_{j =1}^{\infty} x_j 2^{-j}$.  On the
  complement of dyadic points, this map has a well-defined inverse,
  given by the $2$-adic expansion of elements in $[0,1)$. With a few
  modifications, much of the analysis in this paper can also be
  performed on $(\mathbb{X},S)$; compare \cite{BGKS}.
  We have chosen to work with $(\TT,T)$
  because this seems more natural for the applications we have in
  mind, that is, diffraction and hyperuniformity results.  \exend
\end{remark}

The following is a mild adaptation of a result in \cite{Keane}.
  The only difference is that, instead of differentiability, we use
  the weaker summable variation property to obtain equicontinuity of
  the sequence $\{ \varphi_g^n f \}_{n \in \NN}$. The rest of the
  proof remains unchanged. Parts of the following proposition have
also been covered by \cite[Thm.~2.1]{BDEG}.

\begin{prop}[\cite{Keane}]\label{prop:bounded-distortion}
  Let\/ $g \in \cC (\TT)$ be a\/ $g$-function of summable variation,
  and assume that one of the following properties is satisfied, namely
\begin{enumerate}\itemsep=2pt
\item that\/ $g$ has at most one zero in\/ $\TT$,
\item that\/ $g$ has only finitely many zeros in\/ $\TT$, none of
  which wanders into a periodic orbit under the map\/ $T$, or
\item that all zeros of\/ $g$ lie in\/
  $\bigl[ \frac{1}{4}, \frac{3}{4} \bigr)$ or in\/
  $\bigl( \frac{1}{4}, \frac{3}{4}\ts \bigr]$.
\end{enumerate}
Then, for all\/ $f \in \cC (\TT)$, the sequence of functions\/
$(\varphi_g^k f)^{}_{k \in \NN}$ converges, uniformly on\/ $\TT$, to a
constant, denoted by\/ $\mu_g (f)$.  Further, the mapping\/
$f \mapsto \mu_g (f)$ defines a strongly mixing probability measure
on\/ $\TT$, where\/ $\mu_g$ is the unique\/ $g$-measure induced by\/
$g$. \qed
\end{prop}

\begin{remark}
  Proposition~\ref{prop:bounded-distortion} also follows from results
  by Conze and Raugi \cite{ConzeRaugi90}. In fact, under the
  assumption of summable variation, they give a very nice
  characterisation for the uniqueness of $g$-measures and the uniform
  convergence of $\varphi_g^k f$ in terms of \emph{proximality}. Given
  $d \in \NN$, a closed subset $F \subseteq \TT$ is said to be
  \emph{$d$-invariant} if, for each $x \in F$ and $y \in T^{-d} (x)$
  such that $g_n(y) > 0$, it follows that $y \in F$. The function $g$
  is said to be \emph{$d$-proximal} if any two closed $d$-invariant
  subsets of $\TT$ intersect. It was shown in \cite[Secs.~V and
  VI]{ConzeRaugi90} (for the interval $[0,1]$ instead of $\TT$) that a
  $g$-measure is unique precisely if $g$ is $1$-proximal, and that
  $(\varphi^k_g f)_{k \in \NN}$ converges uniformly to a constant if
  and only if $g$ is $d$-proximal for all $d \in \NN$. Indeed, it is
  straightforward to verify that every $g$-function that satisfies any
  of the properties listed in
  Proposition~\ref{prop:bounded-distortion} is in fact $d$-proximal
  for all $d \in \NN$. We can also use this characterisation to give
  weaker conditions under which the conclusion of
  Proposition~\ref{prop:bounded-distortion} holds. For example, we can
  replace $(2)$ and $(3)$ by
\begin{enumerate}
\item[$(2')$] the set $g^{-1}(1)$ is finite and does not contain a
  complete $T$-orbit,
\item[$(3')$] all zeros of $g$ lie in
  $\bigl[ \frac{1}{6}, \frac{5}{6} \bigr)$ or in
  $\bigl( \frac{1}{6}, \frac{5}{6}\ts \bigr]$,
\end{enumerate}
respectively. Since Proposition~\ref{prop:bounded-distortion} in its
present form is sufficient for our purposes, we leave the proof of
this stronger result to the interested reader --- the major ingredients
can be found in \cite{ConzeRaugi90}, some details will be given
in \cite{Diss}. \exend
\end{remark}

With Proposition~\ref{prop:bounded-distortion} at hand,
Theorem~\ref{thm:mug} becomes a direct consequence as follows.

\begin{proof}[Proof of Theorem~$\ref{thm:mug}$]
  Under the assumptions of Proposition~\ref{prop:bounded-distortion},
  we obtain
\[
   \bigl( g^{}_n \leb \bigr) (f) \, = \, \leb(\varphi_g^n f)
   \: \xrightarrow{\, n \to \infty \,} \: \mu_g (f) \ts ,
\]
for all $f \in \cC (\TT)$, where we have used Lemma~\ref{lem:g_k} and
dominated convergence.
\end{proof}

\section{The pure point part of $\mu_g$}\label{sec:pp-of-mu_g}

Next, we want to understand the Lebesgue decomposition
$\mu_g = \mu_{g,\mathsf{pp}} + \mu_{g,\mathsf{sc}} +
\mu_{g,\mathsf{ac}}$.  To this end, let us first summarise some useful
properties of $g$-measures.

\begin{prop}\label{prop:useful}
  The\/ $g$-measure\/ $\mu_g$ of a good\/ $g$-function satisfies the
  following properties.
\begin{enumerate}\itemsep=2pt
\item Either one has\/ $\mu_g = \leb$, or\/ $\mu_g$ is singular with
  respect to\/ $\leb$. In particular, when\/ $\mu_g \ne \leb$, the
  absolutely continuous part of\/ $\mu_g$ vanishes.
\item One has\/ $\mu_g = \leb$ if and only if\/ $g (x) = \frac{1}{2}$
  for a.e.~$x\in\TT$. Within\/ $\cC (\TT)$, this means that\/ $g$ is
  identical to the constant function\/ $\frac{1}{2}$.
\item Either one has\/ $\mu_{g,\mathsf{pp}} = \mu_g$, or\/ $\mu_g$ 
  has no pure point part.
\item The measure\/ $\mu_g$ is always of pure type.   
\end{enumerate}  
\end{prop}

\begin{proof}
Properties  (1) and (2) immediately follow from Theorem~\ref{thm:mug} 
  and the observation made before Theorem~\ref{thm:sing}. 
  To establish (3), we simply verify
  $(\varphi_g)_{*} (\mu_{g, \mathsf{pp}}) =
  ((\varphi_g)_{*}\mu_g)_{\mathsf{pp}} = \mu_{g, \mathsf{pp}}$. Then,
  the uniqueness stated in Proposition~\ref{prop:bounded-distortion}
  gives the claim, and (4) becomes a straightforward consequence.
\end{proof}

More generally, $\mu_g$ is always of pure type whenever it is
  unique for a given $g$-function. This follows from the fact that
  $\mu_{g,\alpha}$ is separately invariant under $(\varphi_g)_{*}$ for
  $\alpha \in \{\mathsf{ac},\mathsf{sc},\mathsf{pp} \} $ if $\mu_g$ is
  invariant \cite[Thm.~1.2]{BDEG}; compare \cite[Lem.~2.2]{DubF66}. 

In what follows, we explore conditions for the existence of a pure
point part. Since some of the observations hold in full generality,
without extra complications with the proofs, we drop the assumption
that $g$ is a good $g$-function for a while. Recall from \cite{Keane}
that the set of $g$-measures coincides precisely with the set of
$T$-invariant measures on $\TT$.  From this, we conclude that the pure
point part of an arbitrary $g$-measure is supported on complete
(forward) $T$-orbits, the latter denoted by
$\cO^{}_{T} (x) \defeq \big\{ T^k x : k \in \NN^{}_0 \big\}$ in what
follows.  In fact, it is also known that $\mu_{g,\mathsf{pp}}$ is
  supported on $g^{-1}(1)$ \cite{ConzeRaugi90}. We give the precise
  statement with a short proof for convenience.

\begin{prop}\label{prop:atom-orbits} 
  Suppose\/ $\mu_g (\{ x \}) > 0$. Then, $x$ is a periodic point of\/
  $T$, and we have\/ $g (y) = 1$ together with\/
  $\mu_g (\{ y \}) = \mu_g (\{ x \})$ for every\/
  $y \in \cO^{}_{T} (x)$.
\end{prop}

\begin{proof}
  Since $\mu_g$ is $T$-invariant, we have
  $\mu_g ( \{ Tx \} ) = \mu_g ( \{ x \} ) + \mu_g ( \{ x + \frac{1}{2}
  \} )$, which implies 
\[
   0 \, < \, \mu_g ( \{ x \} ) \, \leqslant \, \mu_g ( \{ Tx \} ) 
   \, \leqslant \, \cdots \, \leqslant \, \mu_g ( \{ T^n x \} ) 
\]
for all $n \in \NN$. This shows that $\cO^{}_{T} (x)$ must be finite,
and that $\mu_g$ is ultimately constant on this orbit.  Hence, there
exist $j,k \in \NN$ with $j < k$ such that $T^j x = T^k x$.  But then,
we either have $T^{j-1} x = T^{k-1} x$ or
$T^{j-1} x = T^{k-1}x + \frac{1}{2}$. The latter case is impossible
because
\[
    \mu_g ( \{ T^{k-1}x \} ) \, = \, \mu_g ( \{ T^k x \} )  \, = \,
    \mu_g ( \{ T^{k-1} x \} ) + \mu_g ( \{ T^{k-1} x + \tfrac{1}{2} \} )
\]
implies that $\mu_g ( \{ T^{k-1} x + \frac{1}{2} \} ) = 0$, which
contradicts
$\mu_g ( \{ T^{j-1} x \} ) \geqslant \mu_g ( \{ x \} ) > 0$.  We
conclude that $(T^kx)_{k\in\NN_0}$ is periodic and 
$\mu_g$ is constant on $\cO^{}_{T} (x)$.

Since $\mu_g$ is a regular Borel measure, every integrable function
$h$ satisfies
\[
     \mu_g (h) \, = \, \bigl( (\varphi_g)_{*} \mu_g \bigr) (h) 
     \, = \, \mu_g  (\varphi_g h) \ts .
\]
In particular, for $y \in \cO^{}_{T} (x)$ and $h = 1_{\{ y \}}$, we
obtain $\varphi_g 1_{\{ y \}} = g(y) \ts 1_{\{ T y \}}$ and thus
also $\mu_g (1_{\{ y \}}) = g(y) \ts\ts \mu_g (1_{\{T y \}})$. This
implies $g(y) = 1$ for all $y \in \cO^{}_{T} (x)$.
\end{proof}

\begin{remark}
  It is not difficult to verify that some kind of converse of the
  statement in Proposition~\ref{prop:atom-orbits} also holds.  Indeed,
  whenever $x = T^p x$ and $g(T^j x) = 1$ for all
  $0 \leqslant j \leqslant p-1$, the measure
  $\mu = \frac{1}{p} \sum_{j = 0}^{p-1} \delta^{}_{T^j x}$ is a
  $g$-measure.  This observation can be used to construct examples of
  $g$-functions that give rise to more than one $g$-measure; compare
  \cite{ConzeRaugi90, Keane}.  \exend
\end{remark}

Returning to the case of good $g$-functions, we are now ready to prove
Theorem~\ref{thm:sing}.

\begin{proof}[Proof of Theorem~$\ref{thm:sing}$]
  The \textsf{ac} case is clear from properties (1) and (2) of
  Proposition~\ref{prop:useful}.  Next, assume that $\mu_g$ has a 
  non-trivial pure point part and let $x \in \TT$ be such that
  $\mu_g (\{ x \}) > 0$. By Proposition~\ref{prop:atom-orbits}, there
  is some $p \in \NN$ with $T^p x = x$ and $g(T^j x) = 1$ for all
  $0 \leqslant j \leqslant p-1$. We now proceed with a case
  distinction as in Theorem~\ref{thm:mug}.

  First, assume that $g$ has only one zero in $\TT$, say $z$. Then,
  there is precisely one $x\in\TT$ with $g(x) = 1$, namely
  $x=z+\frac{1}{2}$. But this implies $p=1$, hence $T  x = x$, and
  thus $x = 0$. Since this point is unique and $\mu_g$ is a
  probability measure on $\TT$, we get $\mu_g = \delta^{}_{0}$ in this
  case.

  Second, assume that $g$ has only finitely many zeros, none of which
  wanders into a periodic orbit. However, if $\mu_g (\{ x \})>0$, we
  know that $z=x+\frac{1}{2}$ must be a zero of $g$ with $T x = T z$,
  which is impossible because $x$ is a periodic point of $T$, and
  $\mu_{g,\mathsf{pp}} = 0$ in this case.

  Third, assume that all zeros of $g$ lie in
  $\bigl[ \frac{1}{4}, \frac{3}{4} \bigr)$ or in
  $\bigl( \frac{1}{4}, \frac{3}{4} \ts \bigr]$, and let $I$ be either
  of these intervals. If $x\neq 0$, there is some $k \in \NN^{}_{0}$
  such that $T^k x \in I$. But then,
  $z = T^k x + \frac{1}{2} \notin I$ with $g(z) =0$, a contradiction.
  So, we do not get any further location beyond the one from our first
  case.

  Together, this shows that, for a good $g$-function,
  $\mu_g (\{ x \})>0$ is only possible for $x=0$, in which case
  $g \bigl( \frac{1}{2}\bigr) = 0$.  Conversely, assume
  $g \bigl( \frac{1}{2} \bigr) = 0$, which is equivalent to
  $g(0) = 1$. Then, $(\varphi_g f)(0) = f(0)$ for every
  $f \in \cC(\TT)$.  Since $\varphi_g^k f$ converges pointwise to
  $\mu_g (f)$, this yields
\[
    \mu_g(f) \, = \lim_{k \to \infty} \bigl( \varphi_g^k f \bigr) (0) 
    \, = \, f(0) \ts , 
\]
and thus $\mu_g = \delta^{}_0$, which settles the \textsf{pp} case.

Finally, if $\mu_g \neq \mu_{g, \mathsf{pp}}$,
Proposition~\ref{prop:useful} implies that the measure $\mu_g$ cannot
have any pure point part, and is singular relative to $\leb$. The
statement for the \textsf{sc} case is then clear.
\end{proof}

\begin{remark}
  If $\mu$ is a general, $T$-invariant probability measure on $\TT$,
  we can conclude that $\mu_{\mathsf{ac}} = c \ts \leb$ for some
  $c\in [0,1]$. First, it is clear that each of $\mu_{\mathsf{pp}}$,
  $\mu_{\mathsf{sc}}$ and $ \mu_{\mathsf{ac}}$ is separately
  $T$-invariant, so consider $\nu = \mu_{\mathsf{ac}}$. With
  $f^{}_n (x) \defeq \ee^{2\pi \ii n x}$ and a simple calculation, we
  obtain the Fourier--Stieltjes coefficients of $\nu$ as
\[
  \widehat{\nu} (n) \, = \, \nu (f^{}_n) \, = \, \nu (f^{}_n \nts
  \circ T) \, = \, \nu (f^{}_{2n}) \, = \, \widehat{\nu} (2n)
\]
for $n \in \ZZ$, where $\widehat{\nu} (0) = \nu (\TT)$.  Any $n \ne 0$
has a unique representation as $n = 2^k (2m+1)$. Then, by a standard
application of the Riemann--Lebesgue lemma, the above doubling
relation forces $\widehat{\nu} (2m+1) = 0$ for all $m\in\ZZ$ and thus
$\widehat{\nu}(n) = 0$ for all $n \neq 0$. This implies our claim with
$c = \nu (\TT) \in [0,1]$, by the uniqueness of the Fourier--Stieltjes
coefficients.  \exend
\end{remark}

\section{The Distribution Function $F_{g}(x)$}\label{sec:dist}

In this section, we examine the distribution function
$F_g (x) \defeq \mu_g \bigl( [0,x] \bigr)$. We first prove that
$F_g (x)$ is strictly increasing with $x$ if $g$ has at most countably
many zeros. This generalises the classic result
  \cite[Lem.~2.1]{Walters} that $\mu_g$ has full support if $g>0$.
Then, we restrict to $g$-functions with power-law scaling and prove an
effective result in showing that the scaling of $F_g (x)$ near zero is
super-polynomial. In particular, we prove
Theorem~\ref{thm:scaling}. This is also related to the analysis of
hyperuniform structures, and shows that further connections to
number-theoretic questions exist.  In particular, the deviation from a
power-law scaling, as known from the TM sequence (see \cite{GL,BG} and
references therein), is not at all unusual.

\begin{theorem}
  Let\/ $g\in\cC(\TT)$ be a good\/ $g$-function with at most countably
  many zeros, and suppose\/ $g \bigl( \frac{1}{2} \bigr) \neq
  0$. Then, the distribution function\/ $F_{g} (x)$ is strictly
  increasing in\/ $x$.  In particular, every open interval has
  positive\/ $\mu_g$-measure.
\end{theorem}

\begin{proof}
  It suffices to show that $\mu_g(I)>0$ for every interval of the form
  $I = [ 2^{-k} j , 2^{-k} (j+1) ]$, with $k \in \NN$ and
  $0 \leqslant j \leqslant 2^k \! - \nts 1$. Since $\mu_g$ is
  continuous as a measure, weak convergence
  $g^{}_n \xrightarrow{\ts n\to\infty \ts} \ts \mu_g$ yields
\begin{align*}
   \mu_g(I) & \, = \lim_{n \to \infty} \int_I g^{}_n(x) \dd x
   \, =  \lim_{n \to \infty} \int_I   g^{}_k (x)\, 
                g^{}_{n-k} (2^k x) \dd x \\[2mm]
    & \, = \lim_{n \to \infty} \int_0^1 2^{-k}  g^{}_k \bigl(
         2^{-k} ( j + y) \bigr) \, 
         g^{}_{n-k} (y) \dd y \, = \, \mu_g (f) \ts ,
\end{align*}
where $x \mapsto f (x) = 2^{-k} g^{}_k \bigl( 2^{-k} ( j + x)\bigr)$
defines a non-negative function that is bounded by $1$ and has at 
most countably many zeros. Set
$A_n = \big\{ x \in \TT : f (x) \geqslant \frac{1}{n} \big\}$ for
$n \in \NN$ and $B = \{ x \in \TT : f (x) = 0 \}$. Clearly,
$\TT = \bigcup_{n \in \NN} A_n \cup B$. Since $B$ is countable and
$\mu_g$ continuous, we obtain
\[
   1 \, = \, \mu_g (\TT) \, \leqslant \sum_{n \in \NN} \mu_g (A_n) \ts ,
\]
and thus $\mu_g (A_n) > 0$ for some $n \in \NN$. For this choice of
$n$, we find
\[
  \mu_g (I) \, = \, \mu_g (f) \, \geqslant \, \tfrac{1}{n} \ts \mu_g
  (A_n) \, > \, 0 \ts .
\]
Since $k$ and $j$ were arbitrary, this implies that $F_g (x)$ is
strictly increasing in $x$.
\end{proof}

The following result provides inequalities as well as
asymptotics that immediately imply Theorem~\ref{thm:scaling}.

\begin{theorem}\label{thm:eff}
  Let\/ $g\in\cC(\TT)$ be a good\/ $g$-function with power-law
  scaling.  Let\/ $c^{}_1$, $c^{}_2$, $\theta^{}_1$ and\/
  $\theta^{}_2$ be positive constants with\/
  $\theta^{}_2 \leqslant \theta^{}_1$ such that, for all\/
  $x\in \bigl[ 0,\frac{1}{2} \bigr]$,
\begin{equation}\label{eq:lowup}
    c^{}_1 \, x^{\theta_1} \, \leqslant \,
    g(x) \, \leqslant \, c^{}_2 \, x^{\theta_2}.
\end{equation} 
Then, with\/ $s=\min\{1,c^{}_1\}$, $S=\max\{1,c^{}_2\}$ and\/
$\kappa = \mu_g \bigl( \bigl[ \frac{1}{2} , 1 \bigr] \bigr ) >0$, one
has
\[  
   \kappa \ts s \,  2^{-2\theta_1}\, x^{-\frac{\theta_1}{2}\log^{}_2 (x)}\,
  x^{\frac{5\theta_1}{2} - \ts \log^{}_2 (s)} \, < \, F_g (x) \, < \,
  x^{-\frac{\theta_2}{2}\log^{}_2 (x)}\, x^{-\frac{\theta_2}{2} -\log^{}_2 (S)} .
\]
In particular, $F_g (x)$ decays faster than any power of\/ $x$ as\/
$x\to 0^+ \!$.
\end{theorem}

\begin{proof}
  Consider an interval $I_m \defeq [0,2^{-m}]$ for $m \in \NN$. Let
  $\mu^{}_n = g^{}_n \leb$, and recall that the sequence
  $(g^{}_n\leb)_{n \in \NN}$ converges weakly to $\mu_g$ by
  Theorem~\ref{thm:mug}. The inequalities \eqref{eq:lowup} give upper
  and lower bounds on $g^{}_m(x)$ for $x$ close to zero. For
  $y\in[0,1]$, we have
\[
    g^{}_m(2^{-m}y) \, = \, 2^m \prod_{k=0}^{m-1}g(2^{k-m}y) 
    \, \leqslant \,  2^m \prod_{k=0}^{m-1} c^{}_2 \, 2^{(k-m) \theta_2}
    \, = \, (2\ts c^{}_2)^m \,  2^{-\frac{m (m+1)}{2} \ts \theta_2} . 
\]
   Similarly, for the lower bound, we find
\[
    g^{}_m(2^{-m}y) \, = \, 2^m \prod_{k=0}^{m-1}g(2^{k-m}y)
    \, \geqslant \, 2^m \prod_{k=0}^{m-1}c^{}_{1} \,
    2^{(k-m) \ts \theta_1} \ts y^{\theta_1} \, = \, (2\ts c^{}_1)^m\,
    2^{-\frac{m (m+1)}{2}\ts \theta_1 }\ts y^{m\ts\theta_1}. 
\]
We use these bounds on $g^{}_m (2^{-m}y)$ to establish upper and lower
bounds on $\mu_n (I_m)$, and then apply the portmanteau theorem
\cite[Thm.~2.1]{Bill} for $x\in[2^{-(m+1)},2^{-m})$; that is, that
 \[
     \limsup_{n \to \infty} \mu_n (I_{m+1}) \, \leqslant \,
      F_g (x) \, \leqslant \,
     \liminf_{n \to \infty} \mu_n (I_{m}) \ts .
 \]     
For $x\in[2^{-(m+1)},2^{-m})$, we have both
\[
   2^{-(m+1)} \, \leqslant \, x \, < \, 2^{-m}
   \qquad \text{and} \qquad
  -\log^{}_2 (x) - 1 \, \leqslant \, m \, < \, -\log^{}_2 (x) \ts . 
\]
We use these four inequalities freely in what follows.  Further, for
$x\in I_{m+1}$, we have $x< 2^{-m}$, which is the key inequality in
the proof of the lower bound.

Given $n>m$, we use $g^{}_n (x) = g^{}_m (x)\,
 g^{}_{n-m} (2^m x)$ to obtain
\[
\begin{split}
  \mu_n(I_m) \, & = \int_0^{2^{-m}} g^{}_m (x)\,
   g^{}_{n-m}(2^m x) \dd x \, = \, 2^{-m} \int_0^1 
   g^{}_m(2^{-m}y)\, g^{}_{n-m}(y) \dd y \\[2mm]
 &  \leqslant \, c_2^m \, 2^{- \frac{m (m+1)}{2}\ts \theta_2 }
  \int_0^1 g^{}_{n-m}(y) \dd y
  \, = \, c_2^m\, 2^{- \frac{m (m+1)}{2} \ts \theta_2 },
\end{split}
\]
where the last step follows because $g^{}_{n-m}$ is a probability
density on $[0,1]$. Using the last bound and the inequalities from
above, we thus get
\begin{equation}\label{eq:inf}
  \liminf_{n\to\infty}\mu_n(I_m) \, \leqslant \,
  c_2^m\, 2^{-m^2\frac{\theta_2}{2}-m\frac{\theta_2}{2}}
  \, < \,  x^{-\frac{\theta_2}{2}\log^{}_2 (x)}\,
   x^{-\frac{\theta_2}{2} -\log^{}_2 (S)}, 
\end{equation}
where $S=\max\{1, c^{}_2\}$ and thus $\log^{}_{2} (S) \geqslant 0$.

For the lower bound, we obtain
\[
\begin{split}
  \mu_n(I_m) \, & = \, 2^{-m} \int_0^1 g^{}_m (2^{-m}y)\, 
    g^{}_{n-m}(y) \dd y \, \geqslant \, c_1^m\, 
    2^{- \frac{m (m+1)}{2} \ts \theta_1 }
    \int_{0}^1 y^{m\ts \theta_1}\, g^{}_{n-m}(y) \dd y \\[2mm]
   & \geqslant \, c_1^m\, 2^{- \frac{m (m+1)}{2} \ts \theta_1}
 \int_{\frac{1}{2}}^1 2^{- m\ts \theta_1}\, g^{}_{n-m}(y) \dd y
  \: \xrightarrow{n \to \infty} \: \kappa \ts c_1^m\,
   2^{-m^2\frac{\theta_1}{2}-m\frac{3\theta_1}{2}}.
\end{split}
\]
Then, incrementing $m$, we get
\begin{equation}\label{eq:sup}
  \limsup_{n\to\infty} \mu_n(I_{m+1}) \,
  \geqslant \, \kappa \ts c_1^{m+1}\, 2^{-(m+1)^2\frac{\theta_1}{2}
      -(m+1)\frac{3\theta_1}{2}} 
      \, > \, 
    \kappa \ts s \, 2^{-2\ts \theta_1}\, x^{-\frac{\theta_1}{2}\log^{}_2 (x)}\,
    x^{\frac{5\theta_1}{2} - \ts \log^{}_2 (s)} ,
\end{equation}
where $s=\min\{1,c^{}_1\}$ and thus $\log^{}_{2} (s) \leqslant 0$.
Combining \eqref{eq:inf} with \eqref{eq:sup} and applying the
portmanteau theorem as described above provides the desired result.
\end{proof}

Theorem~\ref{thm:eff} has the following consequence.

\begin{corollary}
  Under the assumptions of Theorem~\textnormal{\ref{thm:eff}}, 
  we have
\[
  -\frac{\theta^{}_1}{2} \log^{}_2  (x)^2
  \left(1+O\left(\frac{1}{\log^{}_2 (x) }\right)\right) \,
  \leqslant \, \log \bigl( F_g(x) \bigr) \, \leqslant \,
  -\frac{\theta^{}_2}{2}\log^{}_2 (x)^2
  \left(1+O\left(\frac{1}{\log^{}_2 (x)}\right)\right)
\]  
as\/ $x\to 0^+ \!$. \qed
\end{corollary}

Let us see what this gives for our three guiding examples.

\begin{example}
  Consider our three examples from Example~\ref{ex:guide}.  They all
  share the additional symmetry relation $g(x) = g(1-x)$, for all
  $x \in \bigl[ 0, \frac{1}{2} \bigr]$. This implies the same symmetry
  for $g^{}_n$, for all $n \in \NN$, and hence
  $\kappa = \mu_g \bigl( \bigl[ \frac{1}{2},1 \bigr] \bigr) =
  \frac{1}{2}$.  For the TM measure $\mu_{_\mathrm{TM}}$, we have
  $g^{}_t (x)=\frac{1}{2}\bigl( 1-\cos(2\pi x) \bigr)$. For
  $x\in \bigl[ 0,\frac{1}{2} \bigr]$, one has
  $4x^2\leqslant g^{}_t(x)\leqslant \pi^2 x^2,$ so
  Theorem~\ref{thm:eff} gives
\[
   2^{-5}\, x^{-\log^{}_2 (x)}\, x^{5} \, < \, F_{g^{}_t}(x)
   \, < \,  x^{-\log^{}_2 (x)}\, x^{-1-2\log^{}_2(\pi)},
\]
which should be compared with \cite[Thm.~5.2]{BG}, where a slightly
stronger result was derived by using more specific properties of the
TM measure.

For the tent map, we have $g^{}_\wedge(x)=2\ts x$, so
\[
  2^{-3}\, x^{-\frac{1}{2}\log^{}_2 (x)}\, x^{\frac{5}{2}} \,
  < \, F_{g^{}_\wedge}(x) \, < \,
  x^{-\frac{1}{2} \log^{}_2 (x)}\, x^{-\frac{3}{2}}.
\]
Finally, for the square root map $g^{}_{\nts\sqrt{\phantom{.}}}(x)$,
we get
$\sqrt{x}\leqslant g^{}_{\nts\sqrt{\phantom{.}}}(z) \leqslant
\sqrt{2\ts x}$ for $x\in \bigl[ 0, \frac{1}{2} \bigr]$, hence
\[
  2^{-2}\, x^{-\frac{1}{4}\log^{}_2 (x)}\, x^{\frac{5}{4}} \,
  < \, F_{g^{}_{\nts\sqrt{\phantom{.}}}}(x) \, < \,
  x^{-\frac{1}{4}\log^{}_2 (x)}\, x^{-\frac{3}{4}},
\]
which shows the common structure of all three cases.
\exend
\end{example}
 
It will now be interesting to extend the full scaling analysis of
\cite{BGKS} to these guiding examples, and the general family
treated above.
\bigskip
\bigskip

\section*{Acknowledgements} 

It is a pleasure to thank Uwe Grimm and Neil Ma\~{n}ibo for helpful
discussions, and Gerhard Keller for valuable comments on the topic and
its literature.  We also thank an anonymous referee for pointing
  us to \cite{ConzeRaugi90} and both referees for helpful suggestions
  to improve the presentation. MB and PG acknowledge support by the
German Research Foundation (DFG), via the CRC 1283 at Bielefeld
University. MC and JE acknowledge the support of the Commonwealth of
Australia.

\clearpage


\end{document}